\documentclass[article]{siamart}
\usepackage{color}

\usepackage{lipsum}
\usepackage{amsfonts}
\usepackage{graphicx}
\usepackage{epstopdf}
\usepackage{algorithmic}
\usepackage{enumitem}
\ifpdf
  \DeclareGraphicsExtensions{.eps,.pdf,.png,.jpg}
\else
  \DeclareGraphicsExtensions{.eps}
\fi

\newcommand{\TheTitle}{Minimization based formulations of inverse problems and their regularization} 
\newcommand{\ShortTitle}{Minimization based inverse problems} 
\newcommand{\TheAuthors}{Barbara Kaltenbacher}

\headers{\ShortTitle}{\TheAuthors}

\title{{\TheTitle}\thanks{This work was supported by the Austrian Science Fund FWF under grants I2271 and P30054.}}

\author{Barbara Kaltenbacher
  \thanks{Alpen-Adria-Universit\"at Klagenfurt, Austria
    (\email{barbara.kaltenbacher@aau.at}, \url{wwwu.uni-klu.ac.at/bkaltenb/}).}
}

\usepackage{amsopn}

\theoremstyle{remark}
\newtheorem{example}{Example}
\newtheorem{remark}{Remark}
\newtheorem{assumption}{Assumption}

\newcommand{\R}{\mathbb{R}}
\newcommand{\N}{\mathbb{N}}
\renewcommand{\bar}{\overline}

\newcommand{\calI}{\mathcal{I}}
\newcommand{\calJ}{\mathcal{J}}
\newcommand{\calJKV}{\mathcal{J}^{\rm KV}}
\newcommand{\calD}{\mathcal{D}}
\newcommand{\calQ}{\mathcal{Q}}
\newcommand{\calR}{\mathcal{R}}
\newcommand{\calRt}{\widetilde{\calR}}
\newcommand{\calS}{\mathcal{S}}
\newcommand{\calT}{\mathcal{T}}
\newcommand{\bfF}{\mathbf{F}}

\newcommand{\bfx}{\mathbf{x}}
\newcommand{\bfy}{\mathbf{y}}
\newcommand{\xdag}{x^\dagger}
\newcommand{\udag}{u^\dagger}
\newcommand{\ydel}{y^\delta}
\newcommand{\xad}{x_\alpha^\delta}
\newcommand{\uad}{u_\alpha^\delta}

\newcommand{\ul}[1]{\underline{#1}}
\newcommand{\ol}[1]{\overline{#1}}

\newcommand{\setof}[2]{\left\{#1:#2\right\}}
\newcommand{\argmin}{\mbox{argmin}}
\newcommand{\Mad}{M_{ad}}
\newcommand{\Tikh}{T}
\newcommand{\topo}{\calT}

\newcommand{\trace}{\mathrm{tr}_{\partial\Omega}}
\newcommand{\volt}{\upsilon}
\newcommand{\Volt}{\Upsilon}
\newcommand{\curr}{\gamma}
\newcommand{\Curr}{\Gamma}

\usepackage{graphicx}

\ifpdf
\hypersetup{
  pdftitle={\TheTitle},
  pdfauthor={\TheAuthors}
}
\fi

\def\revision#1{{#1}}

\begin{document}

\maketitle

\begin{abstract}
The conventional way of formulating inverse problems such as identification of a (possibly infinite dimensional) parameter, is via some forward operator, which is the concatenation of the observation operator with the parameter-to-state-map for the underlying model.
Recently, all-at-once formulations have been considered as an alternative to this reduced formulation, avoiding the use of a parameter-to-state map, which would sometimes lead to too restrictive conditions. Here the model and the observation are considered simultaneously as one large system with the state and the parameter as unknowns.
A still more general formulation of inverse problems, containing both the reduced and the all-at-once formulation, but also the well-known and highly versatile so-called variational approach (also called Kohn-Vogelius functional approach) as special cases, is to formulate the inverse problem as a minimization problem -- instead of an equation -- for the state and parameter. Regularization can be incorporated via imposing constraints and/or adding regularization terms to the objective.  
In this paper, after giving a motivation by formulating the electrical impedance tomography (EIT) problem by means of the classical Kohn-Vogelius functional, we will dwell on the regularization aspects  for such variational formulations in an abstract setting. Indeed, combination of regularization by constraints and by penalization leads to new methods that are applicable without solving forward problems. In particular, for the EIT problem we will consider a method employing box constraints in a very natural manner to incorporate the discrepancy principle for regularization parameter choice as well as a priori information on the searched for conductivity.
\end{abstract}



\section{Introduction}
The conventional way of formulating an inverse problem of recovering some quantity $x$ in a model 
\begin{equation}\label{Axu}
A(x,u)=0
\end{equation}
from (indirect) observations of the state $u$
\begin{equation}\label{Cuy}
C(u)=y
\end{equation}
is via an operator equation
\begin{equation}\label{Fxy}
F(x)=y\,,
\end{equation}
where $F=C\circ S$ is the concatenation of the observation operator $C$ with the parameter-to-state-map $S$ defined by 
\begin{equation}\label{AxSx0}
A(x,S(x))=0 \quad \forall x\in\calD\,,
\end{equation}
and thus allowing to eliminate the variable $u$ in \eqref{Axu}, \eqref{Cuy}.
The model operator $A:\widetilde{\calD}\times V\to W$, the observation operator $C:V\to Y$, the forward operator $F:\calD\to Y$ and the parameter-to-state operator $S:\calD\to V$ with $\calD\subseteq \widetilde{\calD}\subseteq X$ act between Hilbert or Banach spaces $V,W,X,Y$. 
Note that well-definedness of $S$ usually requires additional restrictions, so $\calD$ will indeed often be a proper subset of $\widetilde{\calD}$. In the following, we will set $\widetilde{\calD}=X$, as it is in fact often possible.

Recently, all-at-once formulations have been considered as an alternative to \eqref{Fxy} avoiding the use of a parameter-to-state map $S$, which would sometimes lead to too restrictive conditions, see, e.g., \cite{BurgerMuehlhuberIP,BurgerMuehlhuberSINUM,HaAs01,all-at-once,KKV14b,LeHe16}.
They are based on the original formulation as an all-at-once system of model and observation equation \eqref{Axu}, \eqref{Cuy}, 
\begin{equation}\label{Fxuy}
\bfF(\bfx)=\bfF(x,u)=\left(\begin{array}{c}A(x,u)\\C(u)\end{array}\right)
=\left(\begin{array}{c}0\\y\end{array}\right)=\bfy\,.
\end{equation}

A still more general formulation of inverse problems, containing both \eqref{Fxy}, \eqref{Fxuy} but also the well-known variational approach (often also called Kohn-Vogelius functional approach) see, e.g., \cite{Knowles1998,KohnVogelius87,KohnMcKenny90} as well as Example \ref{exKV} below, as special cases, is to formulate the inverse problem as 
\begin{equation}\label{minJ}
(x,u)\in\argmin\setof{\calJ(x,u;y)}{(x,u)\in \Mad(y)}
\end{equation}
for some cost functional $J$ and some admissible set $\Mad(y)$, both of them possibly depending on the data $y$.
Note that in general \eqref{minJ} like \eqref{Fxuy} avoids the existence and use of a parameter-to-state map $S$ \eqref{AxSx0}.
We call this a minimization based formulation (since the term ``variational'' is already occupied in closely related contexts).

Inverse problems in infinite dimensional spaces are often ill-posed and instead of $y$ usually only a noisy version $\ydel$ is available, thus regularization (see, e.g., \cite{BakKok04,EHNbuch96,KalNeuSch08,Kirs96,SGGHL08,SKHK12,Voge02} and the references therein) has to be employed.

Tikhonov regularization can be extended to the formulation \eqref{minJ} in a straightforward manner
\begin{equation}\label{minJR}
(\xad,\uad)\in\argmin\setof{\Tikh_\alpha(x,u;\ydel)=\calJ(x,u;\ydel)+\alpha\cdot\calR(x,u)}{(x,u)\in \Mad^\delta(y^\delta)}\,.
\end{equation}
\revision{
Here the dot in general denotes the Eucildean inner product, as the regularization parameter $\alpha\in\R_+^m$ and the mapping $\calR:X\times V\to\overline{\R}^m$  might have several components,} 
corresponding to several regularization terms or incorporationg different pieces of a priori information on $x$ and $u$. Also the constraints imposed via $\Mad^\delta(\ydel)$ may have a regularizing effect in the spirit of the method of quasi solutions , i.e., Ivanov regularization \cite{KK17, DombrovskajaIvanov65, Ivanov62, Ivanov63, IvanovVasinTanana02, KRR16, LorenzWorliczek13, NeubauerRamlau14, SeidmanVogel89}.
Note that regularization with respect to $u$ might seem unnecessary from the point of view of ill-posedness, which usually only affects $x$; still, certain bounds on the state variable might help to ensure well-posedness of the problem \eqref{minJR} i.e., existence of a minimizer, as will be demonstrated in 
\revision{
Section \ref{subsec:KV_rev} below for a variational formulation of the classical electrical impedance tomography (EIT) problem, cf. \cite{Knowles1998,KohnMcKenny90,KohnVogelius87}.
}

\medskip

Generalization of the convergence analysis to the the setting \eqref{minJ}, \eqref{minJR} will not only allow for an analysis in the sense of regularization theory for the variational approach from Example \ref{exKV} in Sections \ref{subsec:KV}, \ref{subsec:KV_rev} (which to the best of the authors' knowledge has not been carried out so far) but also for the design of new regularizing approaches, e.g, the Morozov type all-at-once version in Subsection \ref{subsec:aaoM}.

\medskip

\revision{
Variational methods have been extensively developed and studied in the context of imaging applications, cf., e.g., \cite{SGGHL08} and the references therein, where usually ill-posedness plays a less pronounced role than here and the purpose of the term $\calR$ is rather to enhance certain image features; note that also $\mathcal{J}$ can be chosen to have several components in order to take into account different types of noise, cf., e.g., \cite{DeLosReyesSchoenlieb}.
We also wish to point to the recent paper \cite{Kindermann}, which relies on minimization based formulations of inverse problems as well; there, the main emphasis is put on the regularized  iterative solution of the resulting minimization problems by gradient type methods.
}

\medskip

The remainder of this paper is organized as follows.
Before proceeding to a convergence analysis for \eqref{minJR} in Section \ref{sec:convanal}, we will demonstrate its generality and applicability by means of 
\revision{
some special cases as well as the EIT problem in Section \ref{sec:cases},
} 
that will be revisited in the second part of Section \ref{sec:convanal}.
Section \ref{sec:convanal} also contains some considerations concerning convexity of the minimization problems \eqref{minJ}, \eqref{minJR}.

\medskip

\paragraph{Notation} By $(\xdag,\udag)$ we denote an exact solution of \eqref{minJ}, which we assume to exist. In the EIT example, this exact solution is denoted by $(\sigma^\dagger,\Phi^\dagger,\Psi^\dagger)$ and we indicate functions with several components such as $\Phi=(\phi_1,\ldots,\phi_I)$ by capital letters.
 
\section{Some special cases \revision{and some inverse problems examples}} \label{sec:cases}

\subsection{Reduced formulation}\label{subsec:red}

Using a discrepancy measure $\calS$ that quantifies the deterministic noise level according to
\begin{equation}\label{delta}
\calS(y,\ydel)\leq\delta\,,
\end{equation}
and the indicator function $\calI_M:W\to\ol{\R}$ defined by $\calI_M(w)=0$ if $w\in M$ and $+\infty$ else,
we can rephrase the conventional inverse problem formulation \eqref{Fxy} as a minimization based one \eqref{minJ}, namely
\begin{equation}\label{minJred}
\left\{\hspace*{-0.1cm}\begin{aligned}
\min_{(x,u)\in X\times V}& \calJ(x,u;y)=\calS(C(u),y)+\calI_{\{0\}}(A(x,u))\\
\mbox{ s.t. }&(x,u)\in\Mad(y)=\calD\times V\,,
\end{aligned}
\right.
\end{equation}
provided $\calS:Y\times Y\to\ol{\R}$ satisfies the definiteness condition
\begin{equation}\label{Sdefinite}
\forall y_1,y_2\in Y\, : \quad \calS(y_1,y_2)\geq0 \quad \mbox{ and }\quad \Bigl(y_1=y_2 \ \Leftrightarrow \ \calS(y_1,y_2)=0\Bigr)\,.
\end{equation}
The discrepancy measure $\calS$ will typically be defined by a Hilbert or Banach space (semi) norm, but can also be a general functional such as the Kullback-Leibler divergence, see, e.g., \cite{DissFlemming,HohageWerner13,DissPoeschl,SGGHL08,DissWerner,Werner15}. 

Alternative ways of rewriting the reduced form \eqref{Fxy} as \eqref{minJ} are, e.g., 
\begin{equation}\label{minJred0}
\left\{\hspace*{-0.1cm}\begin{aligned}
\min_{(x,u)\in X\times V}& \calJ(x,u;y)=\calS(F(x),y)\\
\mbox{ s.t. }&(x,u)\in\Mad(y)=\calD\times \{u_0\}\,,
\end{aligned}
\right.
\end{equation}
with some fixed dummy state $u_0$ or 
\begin{equation}\label{minJred1}
\left\{\hspace*{-0.1cm}\begin{aligned}
\min_{(x,u)\in X\times V}& \calJ(x,u;y)=\calS(C(u),y)\\
\mbox{ s.t. }&(x,u)\in\Mad(y)=\setof{(x,u)\in\calD\times V}{A(x,u)=0}\,.
\end{aligned}
\right.
\end{equation}

To rephrase classical reduced Tikhonov regularization 
\begin{equation}\label{Tikhred}
\min_{x\in\calD} \calS(F(x),\ydel)+\alpha\calR(x)
\end{equation}
as \eqref{minJR}, we return to \eqref{minJred} to immeditately see that \eqref{Tikhred} can be rewritten as \eqref{minJR} with $\calJ$ as in \eqref{minJred},
\[
\calR(x,u)=\calR(x)\,, \quad \Mad^\delta(\ydel)=\calD\times V
\]
using some regularization parameter $\alpha\in\R_+$ and some (usually proper and convex) regularization functional $\calR:\calD\to\ol{\R}$.
Likewise one can write Ivanov regularization \cite{KK17, DombrovskajaIvanov65, Ivanov62, Ivanov63, IvanovVasinTanana02, KRR16, LorenzWorliczek13, NeubauerRamlau14, SeidmanVogel89} as \eqref{minJR} with 
\[
\calR(x,u)=0\,, \quad \Mad^\delta(\ydel)=\setof{x\in\calD}{\calRt(x)\leq\rho}\times V\,,
\]
some functional $\calRt:X\to\ol{\R}$, and $\rho\geq\calRt(\xdag)$, where ideally $\rho=\calRt(\xdag)$, which allows to incorporate a priori known bounds on the exact solution. (For possible other choices of $\rho$ we refer to \cite{DombrovskajaIvanov65,KK17,KRR16}.)

Both approaches can be combined to 
\begin{equation}\label{minJRred}
\left\{\hspace*{-0.1cm}
\begin{aligned}
\min_{(x,u)\in X\times V}& \Tikh_\alpha(x,u;\ydel)=\calS(C(u),y)+\calI_{\{0\}}(A(x,u))+\alpha\cdot\calR(x,u)\\
\mbox{ s.t. }&(x,u)\in \Mad^\delta(\ydel)=\setof{x\in\calD}{\calRt(x)\leq\rho}\times V\,,
\end{aligned}
\right.
\end{equation}
which contains Tihonov and Ivanov regularization as special cases with $\calRt=0$, $\rho=0$, or $\calR=0$, respectively.
\subsection{All-at-once formulation}\label{subsec:aao}
Using a functional $\calQ_A:X\times V\to\ol{\R}$ for quantifying violation of the model, we can similarly rewrite \eqref{Fxuy} as \eqref{minJ} with
\begin{equation}\label{minJaao}
\left\{
\begin{aligned}
\min_{(x,u)\in X\times V}&\calJ(x,u;y)=\calS(C(u),y)+\calQ_A(x,u)\\
\mbox{ s.t. }&(x,u)\in\Mad(y)=X\times V\,,
\end{aligned}
\right.
\end{equation}
under the definiteness conditions \eqref{Sdefinite} and    
\begin{equation}\label{Qdefinite}
\forall x,u\in X\times V\, : \quad \calQ_A(x,u)\geq0 \quad\mbox{ and }\quad\Bigl(A(x,u)=0 \ \Leftrightarrow \ \calQ_A(x,u)=0\Bigr)\,.
\end{equation}
This contains the reduced formulation of the previous subsection as a special case with $\calQ_A(x,u)=\calI_{\{0\}}(A(x,u))$. A further natural instance is $\calQ_A(x,u)=\frac{1}{q}\|A(x,u)\|_W^q$ for some $q\in [1,\infty]$.

All-at-once Tikhonov regularization reads as \eqref{minJR} with 
\[
\Mad^\delta(\ydel)=\Mad(y)=X\times V\,,
\]
where $\alpha\in\R^m_+$ and $\calR: X\times V\to\ol{\R}^m$.

Also here, an Ivanov type version can be defined by setting 
\[
\calR(x,u)=0\,, \quad \Mad^\delta(\ydel)=\setof{(x,u)\in X\times V}{\calRt(x,u)\leq\rho}\,,
\]
with $\calRt:X\times V\to\ol{\R}$, and $\rho\geq\calRt(\xdag,\udag)$.

Again a combination of Tihkonov and Ivanov regularization is given by 
\begin{equation}\label{minJRaao}
\left\{\hspace*{-0.1cm}
\begin{aligned}
\min_{(x,u)\in X\times V}&\Tikh_\alpha(x,u;\ydel)=\calS(C(u),\ydel)+\calQ_A(x,u)+\alpha\cdot\calR(x,u)\\
\mbox{ s.t. }&(x,u)\in \Mad^\delta(\ydel)=\setof{(x,u)\in X\times V}{\calRt(x,u)\leq\rho}\,.
\end{aligned}
\right.
\end{equation}

\subsection{Morozov type all-at-once formulation}\label{subsec:aaoM}
Alternatively to the formulation in Subsection \ref{subsec:aao}, one might consider 
\begin{equation}\label{minJaaoM}
\left\{\hspace*{-0.1cm}
\begin{aligned}
\min_{(x,u)\in X\times V}&\calJ(x,u;y)=\calQ_A(x,u)\\
\mbox{ s.t. }&(x,u)\in\Mad(y)=\setof{(x,u)\in X\times V}{C(u)=y}\,.
\end{aligned}
\right.
\end{equation}

A Tikhonov regularized version reads as \eqref{minJR} with
\[
\Mad^\delta(\ydel)=\setof{(x,u)\in X\times V}{\calS(C(u),\ydel)\le\tau\delta}
\]
for some $\tau\geq1$. We call this approach ``Morozov type'' since it imposes (a multiple of) the noise level as a bound on the discepancy, cf. \cite{Morozov}.
Note that when exchanging the roles of model and observation here
$\calJ(x,u;y)=\calS(C(u),y)$, 
$\Mad(y)=\setof{(x,u)\in X\times V}{A(x,u)=0}=\Mad^\delta(\ydel)$ 
(since the ``model noise level'' is considered as vanishing here)
we would just end up with the reduced formulation \eqref{minJred1} from Subsection \ref{subsec:red}.

An Ivanov-Morozov type all-at-once method results from setting
\[
\calR(x,u)=0\,,\quad \Mad^\delta(\ydel)=\setof{(x,u)\in X\times V}{\calS(C(u),\ydel)\le\tau\delta\mbox{ and }\calRt(x,u)\leq\rho}\,.
\]

Also here, we can combine Tikhonov and Ivanov regularization by considering
\begin{equation}\label{minJRaaoM}
\left\{\hspace*{-0.1cm}
\begin{aligned}
&\min_{(x,u)\in X\times V}\Tikh_\alpha(x,u;\ydel)=\calQ_A(x,u)+\alpha\cdot\calR(x,u)
\\
&\mbox{s.t. }(x,u)\in \Mad^\delta(\ydel)=\setof{(x,u)\in X\times V}{\calS(C(u),\ydel)\le\tau\delta\mbox{ and }\calRt(x,u)\leq\rho}\,.
\end{aligned}
\right.
\end{equation}

\revision{
\subsection{Variational formulations of EIT and other applications}\label{subsec:KV}
The variational approach is highly versatile cf., e.g. \cite{AndrieuxBarangerBenAbda2006,BrownJaisKnowles2005,BrownJais2011,KnowlesWallace1996}, and can, beyond the original electrical impedance tomography EIT application, also be used in different physical contexts, e.g., for the identification of spatially varying Lam\'{e} parameters or for the reconstruction of cracks in linear elasticity from boundary measurements.
\\
We will here go into detail only about the EIT problem in example \ref{exKV} and then shortly sketch two further applications in examples \ref{exKV_mag}, \ref{exKV_crack}.
\begin{example}\label{exKV}
For simplicty of exposition we describe the 2-d case of EIT, following \cite{KohnMcKenny90}.
Here one seeks to identify the distributed conductivity $\sigma:\Omega\to\R$ in
\begin{equation}\label{EIT}
\nabla \cdot J_i= 0\,,\quad \nabla\times E_i=0\,,\quad J_i=\sigma E_i \quad\mbox{ in }\Omega\,,\quad i=1,\ldots, I,
\end{equation}
where $\Omega\subseteq\R^d$, $J_i:\Omega\to\R^d$, $E_i:\Omega\to\R^d$ are the current density and the electric field, respectively, corresponding to a boundary current $j_i=J_i\cdot\nu\vert_{\partial\Omega}$, and leading to a boundary voltage $\volt_i$. 
In 2-d this can be rewritten as 
\[
\nabla \cdot J_i= 0\,,\quad \nabla^\bot\cdot E_i=0\,,\quad J_i=\sigma E_i \quad\mbox{ in }\Omega\,,\quad i=1,\ldots, I,
\]
with the 2-d rotation operator
\[
\nabla^\bot = (-\frac{\partial}{\partial x_2}, \frac{\partial}{\partial x_1})^T\,.
\]
Using potentials $\phi_i$ and $\psi_i$ for $J_i$ and $E_i$   
\[
J_i=-\nabla^\bot \psi_i \,,\quad E_i=-\nabla\phi_i  \,,\quad i=1,\ldots, I\,,
\]
we can rewrite the problem as
\begin{equation}\label{EIT1}
\sqrt{\sigma} \nabla\phi_i=\frac{1}{\sqrt{\sigma}} \nabla^\bot \psi_i \quad \mbox{ in }\Omega\,,
\qquad\psi_i=\curr_i \,,\quad  \phi_i=\volt_i \quad \mbox{ on }\partial\Omega \,,\qquad i=1,\ldots, I\,,
\end{equation}
where for a parametrization $x$ of the boundary $\partial\Omega$ (normalized to $\|\dot{x}\|=1$)
\begin{equation}\label{intj}
\curr_i(x(s))=-\int_0^s j_i(x(r))\, d r\,,
\end{equation}
hence $J_i\cdot\nu=-\nabla^\bot\psi_i\cdot \nu=-\frac{d\curr_i}{ds}=j_i$.
With 
\[
x=\sigma, \quad u=(\Phi,\Psi)=(\phi_1,\ldots,\phi_I,\psi_1,\ldots,\psi_I), \quad y=(\Volt,\Curr)=(\volt_1,\ldots,\volt_I,\curr_1,\ldots,\curr_I),
\] 
this can be cast into the form \eqref{Fxuy} with 
\begin{equation}\label{AWCEIT}
\begin{aligned}
&A(\sigma,\Phi,\Psi)= \left(\sqrt{\sigma}\nabla\phi_1-\tfrac{1}{\sqrt{\sigma}}\nabla^\bot\psi_1,\ldots,\sqrt{\sigma}\nabla\phi_I-\tfrac{1}{\sqrt{\sigma}}\nabla^\bot\psi_I\right)\,,\\ 
& C(\Phi,\Psi)=\trace(\Phi,\Psi)\,,\\
& X\subseteq L^\infty(\Omega)\,, \quad V\subseteq H^1(\Omega)^{2I}\,, \quad W=L^2(\Omega)^{I}\,,  
\end{aligned}
\end{equation}
with the trace operator to be understood component wise.
\\
Using integration by parts and $\nabla\cdot(\nabla^\bot \psi_i)=0$ we have
\begin{equation}\label{intparts}
\int_{\partial\Omega}j_i\volt_i\, ds 
=-\int_{\partial\Omega}\phi_i\nabla^\bot\psi_i\cdot\nu\, ds
=-\int_\Omega\nabla\cdot(\phi_i\nabla^\bot\psi_i)\, dx
=-\int_\Omega\nabla\phi_i\cdot\nabla^\bot\psi_i\, dx\,,
\end{equation}
hence 
\begin{equation}\label{argminAJKV}
\begin{aligned}
&\argmin\setof{\tfrac12\|A(\sigma,\Phi,\Psi)\|_W^2}{C(\Phi,\Psi)=(\Volt,\Curr)}
\\
&=\argmin\setof{\tfrac12\|A(\sigma,\Phi,\Psi)\|_W^2-\int_{\partial\Omega}j_i\volt_i\, ds}{C(\Phi,\Psi)=(\Volt,\Curr)}\\
&=\argmin\setof{\calJKV(\sigma,\Phi,\Psi)}{\trace(\Phi,\Psi)=(\Volt,\Curr)}
\end{aligned}
\end{equation}
with the so-called Kohn-Vogelius functional 
\begin{equation}\label{JEIT}
\calJKV(\sigma,\Phi,\Psi)
=\tfrac12\sum_{i=1}^{I}\int_\Omega \Bigl(\sigma|\nabla \phi_i|^2+\frac{1}{\sigma}|\nabla^\bot \psi_i|^2\Bigr)\, dx
\,.
\end{equation}
With $\calJ=\calJKV$ and the admissible set 
\begin{equation}\label{MadEIT}
\begin{aligned}
\Mad(\Volt,\Curr)=
&\setof{\sigma\in L^\infty(\Omega)}{\sigma>0\mbox{ a.e. in }\Omega}\\
&\times \setof{(\Phi,\Psi)\in H^1(\Omega)^{2I}}{\trace(\Phi,\Psi)=(\Volt,\Curr)}
\end{aligned}
\end{equation}
we arrive at a minimization based formulation \eqref{minJ} of the EIT problem, that goes beyond 
the reduced and all-at-once formulations \eqref{Fxy}, \eqref{Fxuy}.
\\
The minimization of $\calJKV$ over $\Mad(\Volt,\Curr)$ formally leads to the first order optimality conditions
\[
\begin{aligned}
0=&\tfrac12\int_\Omega h\Bigl(\sum_{i=1}^I|\nabla\phi^\dagger_i|^2-\tfrac{1}{{\sigma^\dagger}^2}\sum_{i=1}^I|\nabla^\bot\psi^\dagger_i|^2\Bigr)\, dx \quad \forall h\in L^\infty(\Omega)\\
0=&\int_\Omega\sigma^\dagger \nabla\phi^\dagger_i\cdot\nabla v\,, dx\,, \qquad 0=\int_\Omega\tfrac{1}{\sigma^\dagger} \nabla^\bot\psi^\dagger_i\cdot\nabla^\bot w\,, dx\,, \\
&\forall v,w\in H_0^1(\Omega)\,, \ i\in\{1,\ldots,I\}\,,
\end{aligned}
\]
i.e., $\phi^\dagger_i$ and $\psi^\dagger_i$ solve the weak forms of the elliptic boundary value problems
\begin{equation}\label{PDEsEIT}
\begin{array}{rcl}
-\nabla\cdot(\sigma^\dagger\nabla\phi_i)&=0&\mbox{ in }\Omega\\
\phi_i&=\volt_i&\mbox{ on }\partial\Omega
\end{array} \quad
\begin{array}{rcl}
-\nabla^\bot\cdot(\tfrac{1}{\sigma^\dagger}\nabla^\bot\psi_i)&=0&\mbox{ in }\Omega\\
\psi_i&=\curr_i&\mbox{ on }\partial\Omega\,,
\end{array}
\end{equation}
where the left hand equation corresponds to the classical forward problem of EIT in the reduced formulation \eqref{Fxy} with 
\begin{equation}\label{FyEIT}
F(\sigma)=\trace\Phi\mbox{ where $\phi_i$ solves }\left\{\begin{array}{rcl}
-\nabla\cdot(\sigma^\dagger\nabla\phi_i)&=0&\mbox{ in }\Omega\\
\sigma\frac{\partial\phi_i}{\partial \nu}&=j_i&\mbox{ on }\partial\Omega
\end{array}\right.
\,, \quad y=\Volt\,.
\end{equation}
Note however, that in the all-at-once and minimization based formulations \eqref{Fxuy}, \eqref{minJ}, as opposed to methods based on \eqref{Fxy}, this forward problem will never be solved during minimization of the regularized cost functional and the state $(\Phi,\Psi)$ will reach a solution of this forward problems only in the limit of vanishing noise and regularization.
\\
Equivalence of the all-at-once formulation \eqref{Fxuy}, \eqref{AWCEIT} with the original inverse problem \eqref{EIT1} is obvious. To see equivalence also of the minimization based formulation \eqref{minJ}, \eqref{JEIT}, \eqref{MadEIT} with \eqref{EIT1}, consider the least squares version of \eqref{Fxuy}, \eqref{AWCEIT} (i.e., \eqref{minJ} with $\calJ=\frac12\|A\|_W^2$ and \eqref{MadEIT}) as well as its equivalence to \eqref{minJ}, \eqref{JEIT}, \eqref{MadEIT} via \eqref{argminAJKV}.
\\[2ex]
It is well-known that EIT is a severely ill-posed problem, thus it is necessary to apply regularization.
We first of all do so by imposing upper and lower bounds $\ol{\sigma}>\ul{\sigma}>0$ on the conductivity
\begin{equation}\label{boundssigma}
0\leq\ul{\sigma}\leq \sigma\leq\ol{\sigma}\mbox{ a.e. in }\Omega\,,
\end{equation}
as they are often a priori known in practice.
Aiming at the reconstruction of piecewise constant conductivities, one might use the total variation seminorm
$\calR(\sigma)=\|\sigma\|_{TV}$.
Alternatively, we will here make direct use of the regularizing effect of the pointwise bounds on $\sigma$ in \eqref{boundssigma}, i.e., use Ivanov type $L^\infty$ norm regularization. This
needs additional regularization of the state to guarantee existence of a minimizer, e.g., by using
\begin{equation}\label{REIT}
\calR(\sigma,\Phi,\Psi)=\tfrac12\|(\Phi,\Psi)\|_{{H^{3/2-\epsilon}(\Omega)}^{2I}}^2\,,
\end{equation}
which with $\epsilon\in(0,\frac12)$ admits jumping first derivatives of $\phi_i$, $\psi_i$, hence jumps of $\sigma$, while on the other hand still enabling boundedness (and compactness) of the  embedding $H^{3/2-\epsilon}(\Omega)\to C(\Omega)$.
A possible advantage of the $L^\infty$ instead of $TV$ setting is the easier discretization: The space of piecewise constant Ansatz functions is dense in $L^\infty(\Omega)$, a fact which does not hold in $BV(\Omega)$ cf. Example 4.1 in \cite{BartelsSINUM2012}.
Moreover, the $L^\infty$ constraint after discretization will naturally lead to a minimization with simple box constraints.
\\
To take into account noisy data $(\Volt^\delta,\Curr^\delta)$  with 
\begin{equation}\label{deltaEITb}
\|(\Volt^\delta,\Curr^\delta)-(\Volt,\Curr)\|_{L^\infty(\partial\Omega)^{2I}}\leq\delta
\end{equation}
and guarantee that the exact solution stays admissible for the regularized problem, we use 
\[
\begin{aligned}
\Mad^\delta((\Volt^\delta,\Curr^\delta))=
&\setof{\sigma\in L^\infty(\Omega)}{\ul{\sigma}\leq \sigma\leq\ol{\sigma}\mbox{ a.e. in }\Omega}\\
&\times \setof{(\Phi,\Psi)\in H^1(\Omega)^{2I}}{\|\trace(\Phi,\Psi)-(\Volt^\delta,\Curr^\delta)\|_{L^\infty(\partial\Omega)}\leq \tau\delta}\,,
\end{aligned}
\]
where $\tau>1$ is a fixed safety factor,
to end up  with a purely bound constrained minimization problem
\begin{equation}\label{EITbox}
\left\{\hspace*{-0.1cm}\begin{aligned}
&\min_{\sigma,\Phi,\Psi} 
\sum_{i=1}^{I}\left\{\tfrac12\int_\Omega 
|\sqrt{\sigma}\nabla \phi_i-\frac{1}{\sqrt{\sigma}}\nabla^\bot\psi_i|^2\, dx
+ \tfrac{\alpha}{2}(\|\phi_i\|_{H^{3/2-\epsilon}(\Omega)}^2+\|\psi_i\|_{H^{3/2-\epsilon}(\Omega)}^2)
\right\}\\
&\mbox{s.t.}\quad \ul{\sigma}\leq\sigma\leq\ol{\sigma}\mbox{ on }\Omega\,,\\
&\qquad 
\volt_i^\delta-\tau\delta\leq\phi_i\leq\volt_i^\delta+\tau\delta\,,\
\curr_i^\delta-\tau\delta\leq\psi_i\leq\curr_i^\delta+\tau\delta\,,\
\mbox{ on }\partial\Omega\,, \ i=1,\ldots,I\,,
\end{aligned}
\right.
\end{equation}
where discretization with piecewise constant and piecewise linear Ansatz functions for $\sigma$  and $\phi_i,\psi_i$, respectively, transfers these $L^\infty$ bounds to simple box constraints.
\\
A formulation with $\calJ=\calJKV$ 
\begin{equation}\label{EITJKV}
\left\{\hspace*{-0.1cm}\begin{aligned}
&\min_{\sigma,\Phi,\Psi} 
\sum_{i=1}^{I}\left\{\tfrac12\int_\Omega 
(\sigma|\nabla \phi_i|^2+\frac{1}{\sigma}|\nabla^\bot\psi_i|^2)\, dx
+ \tfrac{\alpha}{2}(\|\phi_i\|_{H^{3/2-\epsilon}(\Omega)}^2+\|\psi_i\|_{H^{3/2-\epsilon}(\Omega)}^2)
\right\}\\
&\mbox{s.t.}\quad \ul{\sigma}\leq\sigma\leq\ol{\sigma}\mbox{ on }\Omega\,,\\
&\qquad 
\volt_i^\delta-\tau\delta\leq\phi_i\leq\volt_i^\delta+\tau\delta\mbox{ on }\partial\Omega\,, \
\|\trace\psi_i-\curr_i^\delta\|_{W^{1,1}(\partial\Omega)}\leq\tau\delta\,,\
 i=1,\ldots,I\,,
\end{aligned}
\right.
\end{equation}
requires to impose the current data $j_i$ in an $L^1$ sense (note that $\curr_i^\delta$ is obtained from the actual measurements $j_i^\delta$ by integration along the boundary \eqref{intj} so that $\|\curr-\curr^\delta\|_{W^{1,1}(\partial\Omega)}\leq \|j-j^\delta\|_{L^1(\partial\Omega)}$)
as our convergence analysis will show,
so that the pure bound constraint structure unfortunatley gets lost.
\end{example}
\begin{remark}
We mention in passing that an alternative way of using the Kohn-Vogelius functional is to  minimize the functional
\[
\sum_{i=1}^{I}\tfrac12\int_\Omega |\sqrt{\sigma}(\nabla \phi_i-\nabla\tilde{\phi}_i)|^2\, dx\,.
\]
under the PDE constraints
\begin{equation}\label{PDEsEIT1}
\begin{array}{rll}
-\nabla\cdot(\sigma\nabla\phi_i)&=0&\mbox{ in }\Omega\\
\phi_i&=\volt_i&\mbox{ on }\partial\Omega
\end{array} \quad
\begin{array}{rcl}
-\nabla\cdot(\sigma\nabla\tilde{\phi}_i)&=0&\mbox{ in }\Omega\\
\sigma\frac{\partial\tilde{\phi}_i}{\partial \nu}&=j_i&\mbox{ on }\partial\Omega\,.
\end{array} \end{equation}
We are not going to consider this approach further here, since via the PDE constraints it involves the parameter-to-state map, which we wish to avoid.
\end{remark}
\begin{remark}
Alternatively, one might think of formulating the inverse problem of EIT as a minimization problem for the Dirichlet energies with appropriately chosen forcing terms defined by $f_i$ in 
\begin{equation}\label{EITDirichletenergy}
\left\{\hspace*{-0.1cm}
\begin{aligned}
&\min_{\sigma,\Phi} 
\tfrac12\sum_{i=1}^{I}\left\{\int_\Omega \sigma|\nabla \phi_i|^2\, dx
+\int_{\partial\Omega} f_i\phi_i\, ds\right\}
&\mbox{s.t. }\phi_i=\volt_i\mbox{ on }\partial\Omega\,, \ i=1,\ldots,I\,.
\end{aligned}
\right.
\end{equation}
This would  only require $I$ instead of $2I$ states and also make the minimization based approach easier extendable to other models based on energy minimization.
However, the optimality conditions for \eqref{EITDirichletenergy} lead to the state equations
\[
\begin{array}{rll}
-\nabla\cdot(\sigma^\dagger\nabla\phi_i)&=0&\mbox{ in }\Omega\\
\sigma\frac{\partial\phi_i}{\partial \nu}&=f_i+\mu_i&\mbox{ on }\partial\Omega\,,
\end{array} 
\]
where $\mu_i$ is the Lagrange multiplier corresponding to the constraint $\phi_i=\volt_i\mbox{ on }\partial\Omega$, and this Lagrange multiplier itself depends on $f_i$ via the remaining equations of the optimality system. Thus it seems involved (if possible at all) to choose $f_i$ such that the observed Neumann boundary conditions $f_i+\mu_i=j_i$ arise.
\end{remark}
}

\begin{remark}\label{rem:exKV-aaoM}
As a matter of fact, formulation \eqref{EITbox} of Example \ref{exKV} can be viewed as a special case of 
\eqref{minJRaaoM} with
\[
\begin{aligned}
&X=L^\infty(\Omega)\,, \quad
Y=L^\infty(\partial\Omega)^{2I}\,,\\  
&V=\setof{(\Phi,\Psi)\in H^1(\Omega)^{2I}}{\trace(\Phi,\Psi)\in Y}\,, \quad
C=\trace(\Phi,\Psi)\,,\\
&\calR(\sigma,\Phi,\Psi)=\tfrac12\|(\Phi,\Psi)\|_{H^{2-\varepsilon}(\Omega)}^2\,,\quad
\calRt(\sigma,\Phi,\Psi)=\|\sigma-\tfrac{\ol{\sigma}+\ul{\sigma}}{2}\|_{L^\infty(\Omega)}\,,\quad 
\rho=\tfrac{\ol{\sigma}-\ul{\sigma}}{2}\,, \\
&\calS((\Volt,\Curr),(\tilde{\Volt},\tilde{\Curr}))=
\|(\Volt,\Curr)-(\tilde{\Volt},\tilde{\Curr})\|_{L^\infty(\partial\Omega)^{2I}}
\end{aligned}
\]
\begin{equation}\label{QAEIT}
\begin{aligned} 
&A(\sigma,\Phi,\Psi)=\left(\sqrt{\sigma}\nabla\phi_1-\tfrac{1}{\sqrt{\sigma}}\nabla^\bot\psi_1,\ldots,\sqrt{\sigma}\nabla\phi_I-\tfrac{1}{\sqrt{\sigma}}\nabla^\bot\psi_I\right)\,, \\
&W=L^2(\Omega)^{I}\,, \quad
\calQ_A(\sigma,\Phi,\Psi)=\frac12\|A(\sigma,\Phi,\Psi)\|_{L^2(\Omega)^{I}}^2\,.
\end{aligned}
\end{equation}
Note that at a first glance this also holds for \eqref{EITJKV} with
\[
\begin{aligned} 
& A(\sigma,\Phi,\Psi)=\left(\sqrt{\sigma}\nabla\phi_1,\ldots,\sqrt{\sigma}\nabla\phi_I,
\tfrac{1}{\sqrt{\sigma}}\nabla^\bot\psi_1,\ldots,\tfrac{1}{\sqrt{\sigma}}\nabla^\bot\psi_I\right)\,, \\
&W=L^2(\Omega)^{2I}\,, \quad
\calQ_A(\sigma,\Phi,\Psi)=\frac12\|A(\sigma,\Phi,\Psi)\|_{L^2(\Omega)^{2I}}^2\,,
\end{aligned}
\]
however this formulation violates the condition $\calQ_A(\sigma^\dagger,\Phi^\dagger,\Psi^\dagger)=0$.

Also note that as opposed to the exact data case (cf. \eqref{argminAJKV}), in the noisy data case $\delta>0$, the two EIT formulations with cost functions $\frac12\|A\|_W$ as in \eqref{QAEIT} and $\calJKV$ as in \eqref{JEIT} are in general not equivalent.
\end{remark}

\revision{
\begin{example}\label{exKV_mag}
In magnetostatics, we have, in place of \eqref{EIT}, the following part of Maxwell's equations in $\R^3$ (note that the typical measurement setup using a so-called Epstein apparatus, i.e., an appropriate arrangement of coils, see, e.g., \cite{BHcurves}, cannot be modeled well in 2-d)
\[
\nabla \cdot B_i= 0\,,\quad \nabla\times H_i=J^{\text{imp}}_i\,,\quad B_i=\mu H_i \quad\mbox{ in }\Omega\,,\quad i=1,\ldots, I.
\]
Here $B_i$, $H_i$ are the magnetic flux density and field, respectively, corresponding to a current $J^{\text{imp}}_i$ impressed by an excitation coil and leading to a magnetic flux $\varPhi_{B,i}=\int_{\Gamma_c}n\cdot B\, d\Gamma$ through a measurement coil, while one typically imposes homogeneous boundary conditions on $B$ and $H$ or rather on their vector and scalar potentials $A_i:\Omega\to\R^3$, $\phi_i:\Omega\to\R$ in 
\[
B_i=\nabla\times A_i \,,\quad H_i=\nabla\phi_i+A_i^J \,,\quad i=1,\ldots, I\,,
\]
where $A^J_i$ is some given -- or at least computable -- vector potential of $J^{\text{imp}}_i=\nabla\times A^J_i$ (which for appropriate domain $\Omega$ exists due to 
$\nabla\cdot J^{\text{imp}}_i=\nabla\cdot(\nabla\times H_i)=0$).
Thus, the inverse problem of identifying a spatially varying permeability $\mu:\Omega\to\R_+$ from measurements of the magnetic fluxes $\varPhi_{B,i}$ for several excitations $J^{\text{imp}}_i$, $i=1,\ldots,I$, can be written as 
\[
\begin{aligned}
&(\mu,A_1,\ldots,A_I,\phi_1,\ldots\phi_I)\in\\
&\argmin\Bigl\{
\frac12\sum_{i=1}^{I}\int_\Omega \left|\sqrt{\mu}(\nabla \phi_i+A^J_i)-\frac{1}{\sqrt{\mu}}\nabla\times A_i\right|^2\, dx\, : \\
&\hspace*{2cm}\phi_i\vert_{\partial\Omega}=0\,, \ n\times A_i\vert_{\partial\Omega}=0\,, \
\oint_{\partial\Gamma_c} A\cdot ds=\varPhi_{B,i}\Bigr\}
\end{aligned}
\]
or, via integration by parts, leading to cancellation of $-\int_\Omega \nabla \phi_i\cdot(\nabla\times A_i)\, dx$, equivalently
\[
\begin{aligned}
&(\mu,A_1,\ldots,A_I,\phi_1,\ldots\phi_I)\in\\
&\argmin\Bigl\{
\sum_{i=1}^{I}\int_\Omega \Bigl(\frac{\mu}{2}|\nabla \phi_i+A^J_i|^2+\frac{1}{2\mu}
|\nabla\times A_i|^2-J^{\text{imp}}_i\cdot A_i\Bigr)\, dx\, :\\
&\hspace*{2cm}\phi_i\vert_{\partial\Omega}=0\,, \ n\times A_i\vert_{\partial\Omega}=0\,, \
\oint_{\partial\Omega_c} A\cdot ds=\varPhi_{B,i}\Bigr\}\,.
\end{aligned}
\]
As opposed to the electrostatic case Example \ref{exKV}, the data $J^{\text{imp}}_i$ here explicitely appears in the cost functional, a possibility that is taken into account in \eqref{minJ}.
\end{example}
\begin{example}\label{exKV_crack}
The detection of cracks from electrostatic measurements, see, e.g., \cite{AndrieuxBenAbda1996,FriedmanVogelius1989}, can be basically directly rephrased in the setting of Example \ref{exKV}, assuming now that the conductivity in the uncracked part of the domain $\Omega\setminus\Sigma$ is known and the searched for quantity is the crack $\Sigma$, which is a curve in the two dimensional domain $\Omega$ (and would be a surface in case $\Omega\subseteq\R^3$).
This can be formulated as  
\begin{equation}\label{minJa_crack}
\begin{aligned}
&(\Sigma,\Phi,\Psi)\in\\
&\argmin\{
\frac12\sum_{i=1}^{I}\int_{\Omega\setminus\Sigma} \left|\sqrt{\sigma}\nabla \phi_i-\frac{1}{\sqrt{\sigma}}\nabla^\bot \psi_i\right|^2\, dx\, : \\
&\hspace*{2cm}\phi_i\vert_{\partial\Omega}=\volt_i\,, \ \psi_i\vert_{\partial\Omega}=\curr_i\,, \ n\cdot\nabla^\bot \psi_i=0\mbox{ on }\Sigma
\}
\end{aligned}
\end{equation}
or equivalently as 
\begin{equation}\label{minJb_crack}
\begin{aligned}
&(\Sigma,\Phi,\Psi)\in\\
&\argmin\{
\sum_{i=1}^{I}\int_{\Omega\setminus\Sigma} \Bigl(\frac{\sigma}{2}|\nabla \phi_i|^2+\frac{1}{2\sigma}|\nabla^\bot \psi_i|^2\Bigr)\, dx\, :\\
&\hspace*{2cm}\phi_i\vert_{\partial\Omega}=\volt_i\,, \ \psi_i\vert_{\partial\Omega}=\curr_i\,.
\end{aligned}
\end{equation}
Via the last constraint in \eqref{minJa_crack} or via the first order optimality conditions in \eqref{minJb_crack}, this imposes homogeneous Neumann boundary conditions for the voltage on $\Sigma$, i.e., an isolating crack. 
\end{example}
}

\section{Convergence analysis} \label{sec:convanal}

In this section we will provide results on well-\-definedness, stability and convergence as the noisy data $\ydel$ tends to the exact one $y$ for the regularization
\begin{equation}\label{minJR1}
(\xad,\uad)\in\argmin\setof{\calJ(x,u;\ydel)+\alpha\cdot\calR(x,u)}{(x,u)\in \Mad^\delta(y^\delta)}\,,
\end{equation}
of the inverse problem
\begin{equation}\label{minJ1}
(x,u)\in\argmin\setof{\calJ(x,u;y)}{(x,u)\in \Mad(y)}
\end{equation}
with an appropriate choice of $\alpha$  and under appropriate conditions.
These are conditions on $\calJ$, $\calR$, the admissible sets $\Mad$, $\Mad^\delta$, and the regularization parameter choice $\alpha$  only.
The functionals $\calQ_A$, $\calS$, $\calRt$, and the operators $A,C,F$ will first of all not appear in this section, only later on, when we exemplarily check the abstract convergence conditions for the special cases from Section \ref{sec:cases}, see Subsection \ref{subsec:cases_rev} below.
Note that the analysis is known for the reduced version \eqref{minJRred} in the two cases $\calRt=0$, $\rho=0$ or $\calR=0$ and for the all-at once Tikhonov version \eqref{minJRaao} with $\calRt=0$, $\rho=0$, but new in all other cases. In particular, 
it also contains a convergence analysis for the variational apporaches from Example \ref{exKV}. 
In this sense, Subsection \ref{subsec:cases_rev} will also serve as an illustration for the first of all quite abstract, seemingly proof-driven assumptions of Subsection \ref{subsec:convana}.

\subsection{Abstract convergence analysis}\label{subsec:convana}

The analysis carried out in this subsection follows the lines of well-known classical results on Tikhonov regularization, see, e.g., \cite{SeidmanVogel89,Tikhonov63} and recent extensions, see, e.g., \cite{DissPoeschl,DissFlemming,DissWerner}, in the sense that the key estimates result from minimality of the regularizer.

\medskip

Well-definedness of a minimizer of \eqref{minJR} for fixed $\alpha,\delta,\ydel$ can be easily established by the direct method of calculus of variations, provided the (generalized) Tikhonov functional
\[
\Tikh_\alpha(x,u;\ydel):=\calJ(x,u;\ydel)+\alpha\cdot\calR(x,u)
\]
and the admissible set $\Mad^\delta(\ydel)$ satisfy the following conditions.
\begin{assumption}\label{asswell}
Let a topology $\topo$ on $X\times V$ exist such that
\begin{enumerate}[label=(\alph*)]
\item\label{asswell:finite}
$\exists (x_0,u_0)\in \Mad^\delta(\ydel)\,:\ \Tikh_\alpha(x_0,u_0;\ydel)<\infty$
\item\label{asswell:compact}
$\forall c\in\R\,:\ M_c:=\setof{(x,u)\in\Mad^\delta(\ydel)}{\Tikh_{\alpha}(x,u;\ydel)\leq c}$ is $\topo$ relatively compact, i.e., every sequence in $M_c$ has a $\topo$ convergent subsequence.
\item\label{asswell:closed} 
$\Mad^\delta(\ydel)$ is $\topo$ closed.
\item\label{asswell:Ttaulsc}
$\Tikh_\alpha(x,u;\ydel)$ is $\topo$-lower semicontinuous, i.e., for any sequence $(x_n,u_n)$ converging with respect to $\topo$ we have \\
$\Tikh_\alpha(\lim^\topo_{n\to\infty}(x_n,u_n);\ydel)\leq \liminf_{n\to\infty}\Tikh_\alpha(x_n,u_n;\ydel)$
\end{enumerate}
\end{assumption}

\begin{proposition}
For each $\ydel\in Y$, $\alpha\in\R_+^m$, under Assumption \ref{asswell}, a minimizer of \eqref{minJR} exits.
\end{proposition}
\medskip

Concerning stability for fixed $\delta$, $\alpha$, a result like 
\[
\begin{aligned}
&y_n\to y^\delta\mbox{ in }Y\quad \Rightarrow\\
&\argmin\setof{\Tikh_\alpha(x,u;y_n)}{(x,u)\in\Mad^\delta(y_n)} \to \argmin\setof{\Tikh_\alpha(x,u;\ydel)}{(x,u)\in\Mad^\delta(\ydel)}
\end{aligned}
\] 
in some sense (cf., e.g., \cite[Theorem 2.1]{EKN89} for reduced Tikhonov regularization) cannot be expected to hold in general, due to the dependence of the admissible sets on $y_n$, which inhibits the use of a minimality argument in this general setting.
However, one can --- besides 
convergence as $\delta\to0$, see Theorem \ref{theo:conv} below --- still achieve uniform boundedness of the minimizers of \eqref{minJR} under the following conditions.
\begin{assumption}\label{assbounded} Let a norm $\|\cdot\|_B$ on $X\times V$ exist such that for all sequences $\forall(y_n)_{n\in\N}\subseteq Y\mbox{ with } y_n\to \ydel\mbox{ in }Y$
\begin{enumerate}[label=(\alph*)]
\item \label{assbounded:x0u0}
$\exists (x_0,u_0)\in X\times V\,\exists n_0\in\N\,\forall n\ge n_0\,: \ (x_0,u_0)\in\Mad^\delta(y_n)$
\item \label{assbounded:Jcont}
$\sup_{(x,u)\in \bigcup_{m\in\N}\Mad^\delta(y_m)}|\calJ(x,u;y_n)-\calJ(x,u,\ydel)|\to0$ as $n\to\infty$.
\item \label{assbounded:Tcoercive}
$\Tikh_\alpha(\cdot,\cdot,\ydel)$ is coercive on the admissible sets, i.e., $C\geq\Tikh_\alpha(x_0,u_0,\ydel)$, $\tilde{C}>0$ exist such that 
$\forall (x_n,u_n)_{n\in\N}\subseteq X\times V\,, \ (x_n,u_n)\in\Mad^\delta(y_n)\,:\ $\\ 
\hspace*{2cm}$\forall\, n\in\N\, : \ \Tikh_\alpha(x_n,u_n,\ydel)\leq C\ \Rightarrow
\ \forall\, n\in\N\, : \ \|(x_n,u_n)\|_B\leq\tilde{C}$
\end{enumerate}
\end{assumption}
Namely, Assumption \ref{assbounded}\ref{assbounded:x0u0},\ref{assbounded:Jcont} implies that for \\
$({\xad}_n,{\uad}_n)\in\argmin\setof{\Tikh_\alpha(x,u;y_n)}{(x,u)\in\Mad^\delta(y_n)}$ and all $n\geq n_0$, 
by minimality 
\[
\begin{aligned}
\limsup_{n\to\infty}\Tikh_\alpha({\xad}_n,{\uad}_n;\ydel)&\leq\limsup_{n\to\infty}\Tikh_\alpha({\xad}_n,{\uad}_n;y_n)
\leq \limsup_{n\to\infty}\Tikh_\alpha(x_0,u_0;y_n)\\
&\leq \Tikh_\alpha(x_0,u_0;\ydel)\leq C\,,
\end{aligned}
\]
hence boundedness of $(\|{\xad}_n,{\uad}_n\|_B)_{n\in\N}$ follows from the coercivity Assumption \ref{assbounded}\ref{assbounded:Tcoercive}.

Note that $\|\cdot\|_B$ might depend on the choice of $\calR$.
E.g., if $\calR$ is defined by some power of a norm, we can choose $\|\cdot\|_B$ as just this norm to satisfy Assumption \ref{assbounded}\ref{assbounded:Tcoercive} provided $\calJ\geq0$.
\begin{proposition} \label{prop:bounded}
For each $\ydel\in Y$, $\alpha\in\R_+^m$, under Assumption \ref{assbounded} the following stability estimate holds
\[
\sup_{n\in\N}\| \setof{\|({\xad}_n,{\uad}_n)\|_B}{({\xad}_n,{\uad}_n)\in\argmin\setof{\Tikh_\alpha(x,u;y_n)}{(x,u)\in\Mad^\delta(y_n)}}
\leq \tilde{C}\,.
\]
\end{proposition}

\medskip

We next show convergence of $(\tilde{x}^\delta,\tilde{u}^\delta):=(x_{\alpha(\delta,\ydel)}^\delta,u_{\alpha(\delta,\ydel)}^\delta)$ to $(\xdag,\udag)$ as $\delta\to0$, with an appropriate choce of the regularization parameter $\alpha=\alpha(\delta,\ydel)$, i.e., we prove the fact that 
\eqref{minJR1} defines a regularization method.

Note that in the context of such minimization based formulations, this regularization property could be abstractly shown by establishing Gamma convergence of the functional   
$\calJ(\cdot,\cdot;\ydel)+\alpha(\delta,\ydel)\cdot\calR(\cdot,\cdot)+\calI_{\Mad^\delta(y^\delta)}$ to $\calJ(\cdot,\cdot;y)+\calI_{\Mad(y)}$ as $\delta\to0$.

Instead, we here provide a direct proof based on the following conditions.
\begin{assumption}\label{assconv}
Let a topology $\topo$ on $ X\times V$ and $\bar{\delta}>0$, $\bar{\alpha}>0$ exist such that 
\begin{enumerate}[label=(\alph*)]
\item\label{assconv:admissible}
$\forall \delta\in(0,\bar{\delta}]\,:\ (\xdag,\udag)\in\Mad^\delta(\ydel)$
\item\label{assconv:finite}
$\forall j\in\{1,\ldots,m\}\,:\ \Bigl(\calR_j(\xdag,\udag)<\infty\mbox{ and }\exists\ul{r}\in\R\,:\ \calR_j\geq\ul{r}\Bigr)$
\item\label{assconv:compact}
$\forall c\in\R\,:\ M_c:=\setof{(x,u)\in\bigcup_{\delta\in(0,\bar{\delta}]}\Mad^\delta(\ydel)}{\Tikh_{\bar{\alpha}}(x,u;y)\leq c}$ is $\topo$ relatively compact.
\item\label{assconv:convM}
$\forall (\delta_n)_{n\in\N}\,, \ (z_n)_{n\in\N}\subseteq Y\,, \ (x_n,u_n)_{n\in\N}\subseteq X\times V \mbox{ with } (x_n,u_n)\in\Mad^{\delta_n}(z_n)\,:$\\$\delta_n\to0\,,\ z_n\to y\,, \ (x_n,u_n)\stackrel{\topo}{\to}(\hat{x},\hat{u})\ \Rightarrow \ (\hat{x},\hat{u})\in \Mad(y)$ 
\item\label{assconv:Jtaulsc}
$\calJ(\cdot,\cdot;y)$ is  $\topo$-lower semicontinuous.
\item\label{assconv:convJ}
$\limsup_{\delta\to0}\sup\setof{\calJ(x,u;y)-\calJ(x,u;\ydel)}{(x,u)\in \bigcup_{d\in(0,\bar{\delta}]}\Mad^d(y^d)}\leq0$\\
\item\label{assconv:convR}
$\forall j\in\{1,\ldots,m\}\,:\ \limsup_{\delta\to0}\frac{1}{\alpha_j(\delta,\ydel)} \Bigl(\calJ(\xdag,\udag;\ydel)-\calJ(\xdag,\udag;y)\Bigr)<\infty$\\
$\forall j\in\{1,\ldots,m\}\,:\ \limsup_{\delta\to0}\frac{1}{\alpha_j(\delta,\ydel)} \Bigl(\calJ(\xdag,\udag;y)-\calJ(\tilde{x}^\delta,\tilde{u}^\delta;\ydel)\Bigr)<\infty$\\
$\alpha(\delta,\ydel)\to0 \mbox{ as }\delta\to0$.
\end{enumerate}
\end{assumption}

\begin{theorem}\label{theo:conv}
Let the operators $\calJ$, $\calR$, the families of data $(\ydel)_{\delta>0}$ and of admissible sets $(\Mad^\delta(\ydel))_{\delta>0}$, and the regularization parameter choice $(\alpha(\delta,\ydel))_{\delta>0}$ satisfy Assumption \ref{assconv}. Then, as $\delta\to0$, $\ydel\to y$, the family $(\tilde{x}^\delta,\tilde{u}^\delta):=(x_{\alpha(\delta,\ydel)}^\delta,u_{\alpha(\delta,\ydel)}^\delta)$ has a $\topo$ convergent subsequence and the limit of every $\topo$ convergent subsequence solves \eqref{minJ1}. If the solution $(\xdag,\udag)$ to \eqref{minJ1} is unique then $(\tilde{x}^\delta,\tilde{u}^\delta)\stackrel{\topo}{\to} (\xdag,\udag)$.
\end{theorem}
\begin{proof}
From minimality of $(\tilde{x}^\delta,\tilde{u}^\delta)$ for \eqref{minJR} and admissibility of $(\xdag,\udag)$ for \eqref{minJR} (Assumption \ref{assconv}\ref{assconv:admissible}) we conclude
\begin{equation}\label{minimality}
\calJ(\tilde{x}^\delta,\tilde{u}^\delta,\ydel)+\alpha(\delta,\ydel)\cdot\calR(\tilde{x}^\delta,\tilde{u}^\delta)\leq
\calJ(\xdag,\udag,\ydel)+\alpha(\delta,\ydel)\cdot\calR(\xdag,\udag)
\end{equation}
i.e., after rearranging and using Assumption \ref{assconv}\ref{assconv:finite}, \ref{assconv:convJ}, \ref{assconv:convR}  
\begin{equation}\label{limsupJ}
\begin{aligned}
&\limsup_{\delta\to0} \calJ(\tilde{x}^\delta,\tilde{u}^\delta,y)-\calJ(\xdag,\udag,y) 
\\
&\leq \limsup_{\delta\to0} \Bigl(\calJ(\tilde{x}^\delta,\tilde{u}^\delta,y)-\calJ(\tilde{x}^\delta,\tilde{u}^\delta,\ydel)
+\calJ(\xdag,\udag,\ydel)-\calJ(\xdag,\udag,y)\\ 
&\qquad\qquad+\alpha(\delta,\ydel)\cdot(\calR(\xdag,\udag)-\calR(\tilde{x}^\delta,\tilde{u}^\delta))
\Bigr)\\ 
&\leq 0\,,
\end{aligned}
\end{equation}
and 
\[
\begin{aligned}
&\sum_{j:\calR_j(\tilde{x}^\delta,\tilde{u}^\delta)-\calR_j(\xdag,\udag)\geq0}\alpha_j(\delta,\ydel)(\calR_j(\tilde{x}^\delta,\tilde{u}^\delta)-\calR_j(\xdag,\udag))\\
&\leq
\alpha(\delta,\ydel)\cdot(\calR(\tilde{x}^\delta,\tilde{u}^\delta)-\calR(\xdag,\udag))
\leq
\calJ(\xdag,\udag,\ydel)-\calJ(\tilde{x}^\delta,\tilde{u}^\delta,\ydel)\,,
\end{aligned}
\]
hence, by Assumption \ref{assconv}\ref{assconv:convJ},\ref{assconv:convR}, we have for all $j\in\{1,\ldots,m\}$
\begin{equation}\label{limsupR}
\begin{aligned}
&\limsup_{\delta\to0} (\calR_j(\tilde{x}^\delta,\tilde{u}^\delta)-\calR_j(\xdag,\udag))
\leq \max\Bigl\{0,\\
&\qquad\limsup_{\delta\to0} \frac{1}{\alpha_j(\delta,\ydel)}
\Bigl(\calJ(\xdag,\udag,\ydel)-\calJ(\xdag,\udag,y)
+\calJ(\xdag,\udag,y)-\calJ(\tilde{x}^\delta,\tilde{u}^\delta,\ydel)\Bigr)\Bigr\}\\
&<\infty\,.
\end{aligned}
\end{equation}
This together with \eqref{limsupJ} by Assumption \ref{assconv}\ref{assconv:finite} implies
\[
\sup_{\delta\in(0,\bar{\bar{\delta}}]} \Tikh_{\bar{\alpha}}(\tilde{x}^\delta,\tilde{u}^\delta,y)<\infty.
\]
for some $\bar{\bar{\delta}}\in(0,\bar{\delta}]$.
Thus by Assumption \ref{assconv}\ref{assconv:compact}, there exists a $\topo$ convergent subsequence $(\tilde{x}_n,\tilde{u}_n)_{n\in\N}\subseteq (\tilde{x}^\delta,\tilde{u}^\delta)_{\delta>0}$. 
The limit $(\hat{x},\hat{u})$ of any such $\topo$ convergent subsequence, by Assumption \ref{assconv}\ref{assconv:convM} lies in the exact admissible set $\Mad(y)$ and due to \eqref{limsupJ} and Assumption \ref{assconv}\ref{assconv:Jtaulsc} satisfies 
\[
\calJ(\hat{x},\hat{u},y)\leq\liminf_{n\to\infty}\calJ(\tilde{x}_n,\tilde{u}_n,y)
\leq \calJ(\xdag,\udag,y)=\min_{(x,u)\in \Mad(y)}\calJ(x,u,y)\,,
\]
thus $(\hat{x},\hat{u})$ solves \eqref{minJ1}.
In case of uniqueness of $(\xdag,\udag)$, a subsequence-subsequence argument yields $\topo$ convergence of the whole sequence $(\tilde{x}^\delta,\tilde{u}^\delta)$ as $\delta\to0$.
\end{proof}
Obviously, Assumption \ref{assconv}\ref{assconv:admissible},\ref{assconv:finite} and Assumption \ref{assbounded}\ref{assbounded:Jcont} (for all $\alpha>0$) imply Assumption \ref{asswell}\ref{asswell:finite}. 
Moreover Assumption \ref{assconv}\ref{assconv:compact},\ref{assconv:convJ} implies Assumption \ref{asswell}\ref{asswell:compact} for $\alpha\in(0,\bar{\alpha}]$.

Thus we can summarize sufficient conditions for well-definedness, stability and convergence as follows:
\begin{assumption}\label{asssum}
Let a topology $\topo$ and a norm $\|\cdot\|_B$ on $X\times V$ and $\bar{\delta}>0$ exist such that 
for the family of noisy data $(\ydel)_{\delta\in(0,\bar{\delta}]}$ satisfying $\ydel\to y$ in $Y$ as $\delta\to0$ and any sequence $(y_n)_{n\in\N}\subseteq Y$ with $y_n\to \ydel$ in $Y$ and for all \\
$(\tilde{x}^\delta,\tilde{u}^\delta)\in\argmin\setof{\calJ(x,u;\ydel)+\alpha(\delta,\ydel)\cdot\calR(x,u)}{(x,u)\in\Mad^\delta(\ydel)}$
\begin{enumerate}[label=(\alph*)]
\item\label{asssum:admissible}
$\forall \delta\in(0,\bar{\delta}]\,\exists n_0\in\N\,\forall n\ge n_0\,:\ (\xdag,\udag)\in\Mad^\delta(\ydel)\cap\Mad^\delta(y_n)$
\item\label{asssum:finite}
$\forall j\in\{1,\ldots,m\}\,:\ \Bigl(\calR_j(\xdag,\udag)<\infty\mbox{ and }\exists\ul{r}\in\R\,:\ \calR_j\geq\ul{r}\Bigr)$
\item\label{asssum:compact}
$\forall c\in\R\,\forall\alpha\in\R_+^m\,:\ M_c:=\setof{(x,u)\in\bigcup_{\delta\in(0,\bar{\delta}]}\Mad^\delta(\ydel)}{\Tikh_{\alpha}(x,u;y)\leq c}$ is $\topo$ relatively compact.
\item \label{asssum:Tcoercive}
$\forall\alpha\in\R_+^m\, \forall \delta\in(0,\bar{\delta}] \,
\exists C_\alpha\geq\Tikh_\alpha(\xdag,\udag,\ydel), \tilde{C}_\alpha>0\\ 
\forall (x_n,u_n)_{n\in\N}\subseteq X\times V\,, \ (x_n,u_n)\in\Mad^\delta(y_n)\,:$\\
\hspace*{2cm}$\forall\, n\in\N\, : \ \Tikh_\alpha(x_n,u_n,\ydel)\leq C_\alpha\ \Rightarrow
\ \forall\, n\in\N\, : \ \|(x_n,u_n)\|_B\leq\tilde{C}_\alpha$
\item\label{asssum:closed} 
$\forall \delta\in(0,\bar{\delta}]\,:\Mad^\delta(\ydel)$ is $\topo$ closed.
\item\label{asssum:convM}
$\forall (\delta_n)_{n\in\N}\,, \ (z_n)_{n\in\N}\subseteq Y\,, \ (x_n,u_n)_{n\in\N}\subseteq X\times V \mbox{ with } (x_n,u_n)\in\Mad^{\delta_n}(z_n)\,:$\\$\delta_n\to0\,,\ z_n\to y\,, \ (x_n,u_n)\stackrel{\topo}{\to}(\hat{x},\hat{u})\ \Rightarrow \ (\hat{x},\hat{u})\in \Mad(y)$ 
\item\label{asssum:Ttaulsc}
$\forall j\in\{1,\ldots,m\}\,:\ \calR_j$ is  $\topo$-lower semicontinuous
\item\label{asssum:Jtaulsc}
$\forall \delta\in(0,\bar{\delta}]$, $\calJ(\cdot,\cdot;\ydel)$ and  $\calJ(\cdot,\cdot;y)$ are  $\topo$-lower semicontinuous.
\item \label{asssum:Jcont}
$\forall \delta\in(0,\bar{\delta}]\,:\ \sup_{(x,u)\in \bigcup_{m\in\N}\Mad^\delta(y_m)}|\calJ(x,u;y_n)-\calJ(x,u,\ydel)|\to0$ as $n\to\infty$
\item\label{asssum:convJ}
$\limsup_{\delta\to0}\sup\setof{\calJ(x,u;y)-\calJ(x,u;\ydel)}{(x,u)\in \bigcup_{d\in(0,\bar{\delta}]}\Mad^d(y^d)}\leq0$\\
\item\label{asssum:convR}
$\limsup_{\delta\to0}\frac{1}{\alpha_j(\delta,\ydel)} \Bigl(\calJ(\xdag,\udag;\ydel)-\calJ(\xdag,\udag;y)\Bigr)<\infty$\\
$\limsup_{\delta\to0}\frac{1}{\alpha_j(\delta,\ydel)} \Bigl(\calJ(\xdag,\udag;y)-\calJ(\tilde{x}^\delta,\tilde{u}^\delta;\ydel)<\infty$\\
$\alpha(\delta,\ydel)\to0 \mbox{ as }\delta\to0$.
\end{enumerate}
\end{assumption}

\subsection{Some special cases \revision{ and EIT }revisited} \label{subsec:cases_rev}

\subsubsection{Reduced, all-at-once and Morozov type all-at-once formulation}\label{subsec:redaaoMaao}

We first of all draw the consequences from the convergence analysis made so far for the reduced \eqref{minJred}, \eqref{minJRred}, the all-at-once \eqref{minJaao}, \eqref{minJRaao}, and the Morozov type all-at-once \eqref{minJaaoM}, \eqref{minJRaaoM} formulation. For this purpose, we make certain assumptions that are  actually quite usual and familiar in the reduced case (see, e.g., \cite{DissFlemming}, \cite{DissPoeschl}, \cite{DissWerner}).

\begin{assumption}\label{ass:aao}
Let a topology $\topo$ and a norm $\|\cdot\|_B$ on $X\times V$ exist such that
\begin{enumerate}[label=(\roman*)]
\item \label{ass:aao:rho}
$\rho\geq \calRt(\xdag,\udag)$
\item \label{ass:aao:finite}
$\forall j\in\{1,\ldots,m\}\,:\ \Bigl(\calR_j(\xdag,\udag)<\infty\mbox{ and }\exists\ul{r}\in\R\,:\ \calR_j\geq\ul{r}\Bigr)$
\item \label{ass:aao:compact} 
for all $z\in Y$, $c>0$, the mapping 
\[
f:X\times V\to\R^4\,, \quad
(x,u)\mapsto (\calQ_A(x,u),\calR(x,u),\calRt(x,u),\calS(C(u),z))
\] 
has $\topo$ compact and $\|\cdot\|_B$ bounded sublevel set 
\[
L_c(f)=\setof{(x,u)\in X\times V}{\max_{i\in\{1,\ldots,4\}} f_i \leq c}
\]
\item \label{ass:aao:lsc} 
for all $z\in Y$, the mapping $f$ defined in the previous item is component wise $\topo$ lower semicontinuous
\item 
\label{ass:aao:Scont}
for all $z\in Y$, the family of mappings $(\calS(z,\cdot):Y\to\R)_{z\in Z}$ is uniformly continuous on $Z=\setof{C(u)}{\exists x\in X\,: \ \calRt(x,u)\leq\rho}$ at $y$, i.e.\\
$\lim_{\hat{y}\to y} \sup_{z\in Z}|\calS(z,\hat{y})-\calS(z,y)|=0$. 
\end{enumerate}
\end{assumption}

\begin{corollary}\label{cor:convredaaoMaao}
Let Assumption \ref{ass:aao} be satisfied and let \eqref{delta}, \eqref{Sdefinite}, \eqref{Qdefinite} hold. (For the reduced formulation we set $\calQ_A(x,u)=\calI(A(x,u))$ and additionally assume $\topo$-closedness of $\calD$.) 

Then for each $\ydel\in Y$, $\alpha\in\R_+^m$, the minimizers of the reduced \eqref{minJRred}, the all-at-once \eqref{minJRaao}, and the Morozov type all-at-once \eqref{minJRaaoM} regularized formulation exist, 
and for any sequence $(y_n)_{n\in\N}\subseteq Y\mbox{ with } y_n\to \ydel\mbox{ in }Y$ as $n\to\infty$ the sequence of corresponding regularized minimizers is $\|\cdot\|_B$ bounded according to Proposition \ref{prop:bounded}.

Assume additionally that a regularization parameter choice satisfying 
\begin{equation}\label{alphared}
\alpha(\delta,\ydel)\to 0\mbox{ and }\frac{\delta}{\revision{\min_{j\in\{1,\ldots,m\}}\alpha_j}(\delta,\ydel)}\leq c\mbox{ as }\delta\to0
\end{equation}
for some $c>0$ independent of $\delta$, is employed, where the second condition in \eqref{alphared} can be skipped for the Morozov type all-at-once formulation.

Then, as $\delta\to0$, $\ydel\to y$, the family $(\tilde{x}^\delta,\tilde{u}^\delta):=(x_{\alpha(\delta,\ydel)}^\delta,u_{\alpha(\delta,\ydel)}^\delta)$ defined by 
the reduced \eqref{minJRred}, the all-at-once \eqref{minJRaao}, and the Morozov type all-at-once \eqref{minJRaaoM} regularized formulation, 
has a $\topo$ convergent subsequence and the limit of every $\topo$ convergent subsequence solves the inverse problem with exact data \eqref{Axu}, \eqref{Cuy}. If the solution $(\xdag,\udag)$ to \eqref{Axu}, \eqref{Cuy} is unique, then $(\tilde{x}^\delta,\tilde{u}^\delta)\stackrel{\topo}{\to} (\xdag,\udag)$.
\end{corollary}

\begin{proof}
We step by step check all items of Assumption \ref{asssum} for the all-at-once formulation (aao), (which contains the reduced one as a special case), and for the Morozov type all-at-once formulation (Maao):
\begin{itemize}
\item[\ref{asssum:admissible}]
\begin{itemize}
\item in case of (aao)
follows from Assumption \ref{ass:aao}\ref{ass:aao:rho}
\item in case of (Maao) follows from \eqref{delta}, $\tau\geq1$ and Assumption \ref{ass:aao}\ref{ass:aao:rho}. With $\tau>1$ and Assumption \ref{ass:aao}\ref{ass:aao:Scont} we also get $(\xdag,\udag)\in \Mad^\delta(y_n)$ for $n$ sufficiently large.
\end{itemize}
\item[\ref{asssum:finite}]
is explicitely imposed in Assumption \ref{ass:aao}\ref{ass:aao:finite}.
\item[\ref{asssum:compact},\ref{asssum:Tcoercive}]
follow from Assumption \ref{ass:aao}\ref{ass:aao:compact}.
\item[\ref{asssum:closed}]
\begin{itemize}
\item in case of (aao)
follows from $\topo$ lower semi continuity of $\calRt$ in Assumption \ref{ass:aao}\ref{ass:aao:lsc}
\item in case of (Maao)
follows from $\topo$ lower semicontinuity of $\calRt$ and of $(x,u)\mapsto\calS(C(u),\ydel)$ according to Assumption \ref{ass:aao}\ref{ass:aao:lsc}.
\end{itemize}
\item[\ref{asssum:convM}]
\begin{itemize}
\item in case of (aao)
is trivial for $\Mad(y)=X\times V$; note that in the reduced case we explicitely assume $\topo$-closedness of $\calD$.
\item in case of (Maao) 
can be concluded by Assumption \ref{ass:aao}\ref{ass:aao:rho} and using Assumption \ref{ass:aao}\ref{ass:aao:lsc}\ref{ass:aao:Scont} to obtain
\[
\begin{aligned}
&\calS(C(\hat{u}),y)\leq\liminf_{n\to\infty} \calS(C(u_n),y)\\
&=\liminf_{n\to\infty} \Bigl(\calS(C(u_n),y_n)  +\calS(C(u_n),y)-\calS(C(u_n),y_n)\Bigr)\\
&\leq\liminf_{n\to\infty} \Bigl(\tau\delta_n  +\calS(C(u_n),y)-\calS(C(u_n),y_n)\Bigr)=0\,,
\end{aligned}
\]
which by definiteness of $\calS$ \eqref{Sdefinite} implies $C(\hat{u})=y$.
\end{itemize}
\item[\ref{asssum:Ttaulsc},\ref{asssum:Jtaulsc}]
are explicitely imposed in Assumption \ref{ass:aao}\ref{ass:aao:lsc}
\item[\ref{asssum:Jcont},\ref{asssum:convJ}]
\begin{itemize}
\item in case of (aao)
are guaranteed by Assumption \ref{ass:aao}\ref{ass:aao:Scont}
\item in case of (Maao)
hold automatically, since $\calJ(x,u,y)=\calQ_A(x,u)$ does not depend on $y$ here
\end{itemize}
\item[\ref{asssum:convR}]
\begin{itemize}
\item in case of (aao)
follows from
\[
\calJ(\xdag,\udag;\ydel)-\calJ(\xdag,\udag;y)=\calS(C(\udag),\ydel)-0\leq\delta
\] 
by \eqref{delta} and from
\[
\calJ(\xdag,\udag;y)-\calJ(\tilde{x}^\delta,\tilde{u}^\delta;\ydel)=0-\calJ(\tilde{x}^\delta,\tilde{u}^\delta;\ydel)\leq0
\]
by \eqref{Sdefinite}, \eqref{Qdefinite}, as well as \eqref{alphared}.
\item in case of (Maao), due to  independence of $J$ on $y$ we only have to verify the second condition, which immediately follows from 
\[
\calJ(\xdag,\udag;y)-\calJ(\tilde{x}^\delta,\tilde{u}^\delta;\ydel)
=0-\calQ_A(\tilde{x}^\delta,\tilde{u}^\delta)\leq0\,.
\]
by \eqref{Qdefinite}.
\end{itemize}
\end{itemize}
\end{proof}

\subsubsection{Variational approach for EIT}\label{subsec:KV_rev}

Recall the setting of Example \ref{exKV} 
\begin{equation}\label{settingEIT}
\begin{aligned}
&x=\sigma\,,\
u=(\Phi,\Psi)=(\phi_1,\ldots,\phi_I,\psi_1,\ldots,\psi_I)\,,\
y=(\Volt,\Curr)=(\volt_1,\ldots,\volt_I,\curr_1,\ldots,\curr_I)\,,\\
&X=L^\infty(\Omega)\,, \quad V=\setof{(\Phi,\Psi)\in H^1(\Omega)^{2I}}{\trace(\Phi,\Psi)\in Y}\,, \\
&\calR(\sigma,\Phi,\Psi)=\tfrac12\|(\Phi,\Psi)\|_{{H^{3/2-\epsilon}(\Omega)}^{2I}}^2\,,\quad
\calRt(\sigma,\Phi,\Psi)=\|\sigma-\tfrac{\ol{\sigma}+\ul{\sigma}}{2}\|_{L^\infty(\Omega)}\,,\quad \rho=\tfrac{\ol{\sigma}-\ul{\sigma}}{2}\,,
\end{aligned}
\end{equation}

We first of all consider the all-at-once version written as a minimization problem \eqref{EITbox}, i.e., \eqref{minJR} with 
\begin{equation}\label{JbMad}
\begin{aligned}
&\calJ(\sigma,\Phi,\Psi)=
\sum_{i=1}^{I}\tfrac12\int_\Omega |\sqrt{\sigma}\nabla \phi_i-\frac{1}{\sqrt{\sigma}}\nabla^\bot\psi_i|^2\, dx
\\
&\Mad^\delta((\Volt^\delta,\Curr^\delta))
=\setof{\sigma\in L^\infty(\Omega)}{\ul{\sigma}\leq \sigma\leq\ol{\sigma}\mbox{ a.e. in }\Omega}\\
&\hspace*{2.5cm}\times \setof{(\Phi,\Psi)\in H^1(\Omega)^{2I}}{\|\trace(\Phi,\Psi)-(\Volt^\delta,\Curr^\delta)\|_{L^\infty(\partial\Omega)}\leq \tau\delta}\,,
\end{aligned}
\end{equation}
with $\epsilon\in(0,\frac12)$ (where a total variation regularization term $\|\sigma\|_{TV}$ may or may not be added), which according to Remark \ref{rem:exKV-aaoM} can be treated as a special case of Subsection \ref{subsec:redaaoMaao} with
\[
\begin{aligned} 
&A(\sigma,\Phi,\Psi)=\left(\sqrt{\sigma}\nabla\phi_1-\tfrac{1}{\sqrt{\sigma}}\nabla^\bot\psi_1,\ldots,\sqrt{\sigma}\nabla\phi_I-\tfrac{1}{\sqrt{\sigma}}\nabla^\bot\psi_I\right)\,, \\
&W=L^2(\Omega)^{I}\,, \quad
\calQ_A(\sigma,\Phi,\Psi)=\frac12\|A(\sigma,\Phi,\Psi)\|_{L^2(\Omega)^{I}}^2\,,\\
&C=\trace(\Phi,\Psi)\,, \quad Y=L^\infty(\partial\Omega)^{2I}\,, \quad 
\calS(y_1,y_2)=\|y_1-y_2\|_Y\,.
\end{aligned}
\]
The reason for using the $L^\infty$ discrepancy is the computationally advantageous formulation as a bound constrained minimization problem \eqref{EITbox}.

The topology needed in Assumption \ref{ass:aao} can be defined by 
\begin{equation}\label{topoEIT}
(\sigma_n,\Phi_n,\Psi_n)\stackrel{\topo}{\to}(\sigma,\Phi,\Psi)\ \Leftrightarrow \
\begin{cases}
\sigma_n\stackrel{*}{\rightharpoonup}\sigma \mbox{ and }\frac{1}{\sigma_n}\stackrel{*}{\rightharpoonup}\frac{1}{\sigma} \mbox{ in }L^\infty(\Omega)\,, \\
(\Phi_n,\Psi_n)\to(\Phi,\Psi) \mbox{ in }H^{2-3\epsilon/2}(\Omega)^{2I}\\
(\Phi_n,\Psi_n)\rightharpoonup(\Phi,\Psi) \mbox{ in }H^{3/2-\epsilon}(\Omega)^{2I}\,, \\
\trace(\Phi_n,\Psi_n)\to\trace(\Phi,\Psi) \mbox{ in }L^\infty(\partial\Omega)^{2I}
\end{cases}
\end{equation}
and the norm in which stability can be achieved, is obviously given by
\begin{equation}\label{normBEIT}
\|(\sigma,\Phi,\Psi)\|_B:=
\|\sigma\|_{L^\infty(\Omega)}
+\|\Phi,\Psi)\|_{H^{3/2-\epsilon}(\Omega)^{2I}}^2
+\|\trace(\Phi,\Psi)\|_{L^\infty(\partial\Omega)^{2I}}\,.
\end{equation}

Therewith, one can now verify Assumption \ref{ass:aao}, presuming
\begin{equation}\label{sigmadagger}
\ul{\sigma}\leq\sigma^\dagger\leq\ol{\sigma}\mbox{ a.e. in }\Omega
\end{equation}
and
\begin{equation}\label{assconv:finiteEIT}
(\Phi^\dagger,\Psi^\dagger)\in {H^{3/2-\epsilon}(\Omega)}^{2I}\,,
\end{equation}
which leads to the following result.
\begin{corollary}
Let \eqref{sigmadagger}, \eqref{assconv:finiteEIT}, \eqref{delta} hold. 

Then for any $\alpha>0$, a minimizer of \eqref{minJR} with \eqref{JbMad} exists
and for any sequence $((\Volt_n,\Curr_n))_{n\in\N}\subseteq Y\mbox{ with } (\Volt_n,\Curr_n)\to (\Volt^\delta,\Curr^\delta)\mbox{ in }Y$ as $n\to\infty$, the sequence of corresponding regularized minimizers is $\|\cdot\|_B$ bounded, for $\|\cdot\|_B$ as in \eqref{normBEIT}.

Assume additionally that a regularization parameter choice satisfying 
$\alpha(\delta,\ydel)\to 0$
as $\delta\to0$
is employed.

Then, as $\delta\to0$, $(\Volt^\delta,\Curr^\delta)\to (\Volt,\Curr)$ in Y, the family $(\tilde{\sigma}^\delta,\tilde{\Phi}^\delta,\tilde{\Psi}^\delta):=$ \\$(\sigma_{\alpha(\delta,\ydel)}^\delta,\Phi_{\alpha(\delta,\ydel)}^\delta,\Psi_{\alpha(\delta,\ydel)}^\delta)$
has a $\topo$ convergent subsequence and the limit of every $\topo$ convergent subsequence solves the inverse problem with exact data \eqref{Fxuy}. 
\end{corollary}
\begin{proof} 
First of all, we point out that for the Morozov type all-at-once formulation relevant here, the second condition in \eqref{alphared} is not required, but $\alpha(\delta)$ has to be strictly positive to guarantee existence of minimizers of the regularized problems.

Now we verify each item of Assumption \ref{ass:aao}.
\begin{itemize}
\item[\ref{ass:aao:rho}] 
follows from \eqref{sigmadagger}
\item[\ref{ass:aao:finite}]
follows from \eqref{assconv:finiteEIT} and nonnegativity of $\calR$
\item[\ref{ass:aao:compact}] Boundedness  of $\calR(\sigma,\Phi,\Psi)$, $\calRt(\sigma,\Phi,\Psi)$, and $\calS(C(\Phi,\Psi),y)$ obviously implies boundedness of $\|(\sigma,\Phi,\Psi)\|_B$.

By compactness of the embeddings $H^{3/2-\epsilon}(\Omega)\hookrightarrow H^{(4-2\epsilon)/3}(\Omega)$,
and continuity of the trace and embedding operators 
$H^{(4-2\epsilon)/3}(\Omega)$ $\to$ $H^{(5-4\epsilon)/6}(\partial\Omega)$ $\to$ $L^\infty(\partial\Omega)$ for $\epsilon\in(0,\frac12)$ the boundedness $\|(\Phi_n,\Psi_n)\|_{H^{3/2-\epsilon}(\Omega)^{2I}}^2\leq c$ implies strong convergence of 
$\phi_{n,i}$ and $\psi_{n,i}$ in $H^{(4-2\epsilon)/3}(\Omega)$ and of 
$\trace\phi_{n,i}$ and $\trace\psi_{n,i}$ in $L^\infty(\partial\Omega)$
along a subsequence.
\item[\ref{ass:aao:lsc}] 
$\topo$ lower semicontinuity of 
\begin{itemize}
\item[$\bullet$] $\calQ_A$ --- actually even $\topo$ continuity --- follows from 
\[
\begin{aligned}
&|\calQ_A(\sigma_n,\Phi_n,\Psi_n)-\calQ_A(\hat{\sigma},\hat{\Phi},\hat{\Psi})|\\
&=
|\sum_{i=1}^{I}\tfrac12\int_\Omega \Bigl(
|\sqrt{\sigma_n}\nabla \phi_{n,i}-\frac{1}{\sigma_n}\nabla^\bot \psi_{n,i}|^2
-|\hat{\sigma}\nabla \hat{\phi}_i-\frac{1}{\hat{\sigma}}\nabla^\bot \hat{\psi}_i|^2\Bigr)\, dx|\\
&=|\sum_{i=1}^{I}\tfrac12\int_\Omega \Bigl(
\sigma_n|\nabla \phi_{n,i}|^2-\hat{\sigma}|\nabla \hat{\phi}_i|^2
+\frac{1}{\sigma_n}|\nabla^\bot \psi_{n,i}|^2-\frac{1}{\hat{\sigma}}|\nabla^\bot \hat{\psi}_i|^2\\
&\qquad\qquad-2\nabla\phi_{n,i}\cdot\nabla^\bot\psi_{n,i}+2\nabla\hat{\phi}\cdot\nabla^\bot\hat{\psi}\Bigr)\, dx|
\end{aligned}
\]
where from $\topo$ convergence of $(\sigma_n,\Phi_n,\Psi_n)$ to $(\hat{\sigma},\hat{\Phi},\hat{\Psi})$ we can conclude
\begin{equation}\label{estJb}
\hspace*{-0.5cm}
\begin{aligned}
&|\int_\Omega \hat{\sigma}|\nabla \hat{\phi}_i|^2-\sigma_n|\nabla \phi_{n,i}|^2|\\
&\leq |\int_\Omega \Bigl(\hat{\sigma}-\sigma_n)|\nabla \hat{\phi}_i|^2\, dx|
+|\int_\Omega \sigma_n(|\nabla \hat{\phi}_i|^2-|\nabla \phi_{n,i}|^2)\, dx|\\ 
&\leq |\int_\Omega \Bigl(\hat{\sigma}-\sigma_n)|\nabla \hat{\phi}_i|^2\, dx|
+\overline{\sigma}(\|\hat{\phi}_i\|_{H^1(\Omega)}+\|\phi_{n,i}\|_{H^1(\Omega)})\|\hat{\phi}_i-\phi_{n,i}\|_{H^1(\Omega)})\\
&\to0 \mbox{ as }n\to\infty
\end{aligned}
\end{equation}
since $|\nabla \hat{\phi}_i|^2\in L^1(\Omega)$, and analogously for the $\frac{1}{\sigma}\nabla^\bot\psi$ part, as well as
\[
\hspace*{-0.5cm}
\begin{aligned}
&|\int_\Omega (\nabla\phi_{n,i}\cdot\nabla^\bot\psi_{n,i}
-\nabla\hat{\phi}\cdot\nabla^\bot\hat{\psi})\, dx|\\
&\leq\int_\Omega \nabla\phi_{n,i}\cdot(\nabla^\bot\psi_{n,i}-\nabla^\bot\hat{\psi})
+(\nabla\phi_{n,i}-\nabla\hat{\phi})\cdot\nabla^\bot\hat{\psi})\, dx|\\
&\leq 
\|\nabla\phi_{n,i}\|_{L^2(\Omega)}\| \nabla^\bot\psi_{n,i}-\nabla^\bot\hat{\psi}\|_{L^2(\Omega)}
+\|\nabla\phi_{n,i}-\nabla\hat{\phi}\|_{L^2(\Omega)}\|\nabla^\bot\hat{\psi}\|_{L^2(\Omega)}\\
&\to0 \mbox{ as }n\to\infty
\end{aligned}
\]
\item[$\bullet$] 
$\calR=\|\cdot\|_{H^{3/2-\epsilon}(\Omega)^{2I}}^2$ is obvious.
\item[$\bullet$] $\calRt$ follows from weak* lower semicontinuity of the $L^\infty$ norm 
\item[$\bullet$] $(x,u)\mapsto\calS(C(u),z)$ follows from continuity of the $L^\infty(\partial\Omega)$ norm and strong convergence of the traces according to the last line of \eqref{topoEIT}
\end{itemize}
\item [\ref{ass:aao:Scont}] 
follows from the reverse triangle inequality in $L^\infty(\partial\Omega)$:
\[
|\calS((\Volt,\Curr),(\Volt^1,\Curr^1))
-\calS((\Volt,\Curr),(\Volt^2,\Curr^2))|
\leq \|(\Volt^1,\Curr^1)-(\Volt^2,\Curr^2)\|_{L^\infty(\partial\Omega)^{2I}}
\] 
\end{itemize}
\end{proof}

\medskip

For the formulation based on the Kohn-Volgelius functional we use \eqref{settingEIT} with
\begin{equation}\label{JaMad}
\begin{aligned}
&\calJKV(\sigma,\Phi,\Psi)
=\tfrac12\sum_{i=1}^{I}\int_\Omega \Bigl(\sigma|\nabla \phi_i|^2+\frac{1}{\sigma}|\nabla^\bot \psi_i|^2\Bigr)\, dx\,,
\\
&\Mad^\delta((\Volt^\delta,\Curr^\delta))
=\setof{\sigma\in L^\infty(\Omega)}{\ul{\sigma}\leq \sigma\leq\ol{\sigma}\mbox{ a.e. in }\Omega}\\
&\hspace*{2.5cm}\times \setof{(\Phi,\Psi)\in H^1(\Omega)^{2I}}{\|\trace(\Phi,\Psi)-(\Volt^\delta,\Curr^\delta)\|_Y\leq \tau\delta}\,,
\end{aligned}
\end{equation}
with the slightly stronger space and misfit functional
\begin{equation}\label{YJKV} 
Y=L^\infty(\partial\Omega)^{I}\times W^{1,1}(\partial\Omega)^I\,, \quad 
\calS(y_1,y_2)=\|y_1-y_2\|_Y\,,
\end{equation}
(note that \eqref{delta} with $\calS$, $Y$ as in \eqref{YJKV} via the relation \eqref{intj} corresponds to an $L^1$ noise level on the observe currents $j_i$), 
as well as condition \eqref{alphared} on the regularization parameter choice, and directly verify Assumption \ref{asssum}, which yields the following result.
\begin{corollary}
Let \eqref{sigmadagger}, \eqref{assconv:finiteEIT}, \eqref{delta} hold. 

Then for any $\alpha>0$, a minimizer of \eqref{minJR} with \eqref{JaMad} exists
and for any sequence $((\Volt_n,\Curr_n))_{n\in\N}\subseteq Y\mbox{ with } (\Volt_n,\Curr_n)\to (\Volt^\delta,\Curr^\delta)\mbox{ in }Y$ as $n\to\infty$, the sequence of corresponding regularized minimizers is $\|\cdot\|_B$ bounded, for $\|\cdot\|_B$ as in \eqref{normBEIT}.

Assume additionally that a regularization parameter choice satisfying 
$\alpha(\delta,\ydel)\to 0$
and $\frac{\delta}{\alpha(\delta,\ydel)}\leq c$
as $\delta\to0$
for some $c>0$ independent of $\delta$, 
is employed.

Then, as $\delta\to0$, $(\Volt^\delta,\Curr^\delta)\to (\Volt,\Curr)$, the family $(\tilde{\sigma}^\delta,\tilde{\Phi}^\delta,\tilde{\Psi}^\delta):=$\\ $(\sigma_{\alpha(\delta,\ydel)}^\delta,\Phi_{\alpha(\delta,\ydel)}^\delta,\Psi_{\alpha(\delta,\ydel)}^\delta)$
has a $\topo$ convergent subsequence and the limit of every $\topo$ convergent subsequence solves the inverse problem with exact data \eqref{Fxuy}. 
\end{corollary}
\begin{proof}

\begin{itemize}
\item[\ref{asssum:admissible}]
follows from \eqref{delta}, $\tau>1$ and \eqref{sigmadagger}.
\item[\ref{asssum:finite}]
is explicitely imposed by \eqref{assconv:finiteEIT}.
\item[\ref{asssum:compact},\ref{asssum:Tcoercive}]
follow like in case of $\calJ=\calQ_A$ by definition of $\topo$.
\item[\ref{asssum:closed}] 
follows from weak* lower semicontinuity of the $L^\infty(\Omega)$ and the $Y$ norm.
\item[\ref{asssum:convM}]
follows from \eqref{sigmadagger} and the triangle inequality in $L^\infty(\partial\Omega)$, which implies 
\[
\begin{aligned}
&\|\trace\hat{\phi}_i-\volt_i\|_{L^\infty}\\
&\leq
\|\trace\hat{\phi}_i-\trace\phi_{n,i}\|_{L^\infty}
+\|\trace\phi_{n,i}-\volt_{n,i}\|_{L^\infty}
+\|\volt_{n,i}-\volt_i\|_{L^\infty}\\ 
&\to0\mbox{ as }n\to\infty
\end{aligned}
\]
and likewise for $\|\trace\hat{\psi}_i-\curr_i\|_{W^{1,1}}$.
\item[\ref{asssum:Ttaulsc}] follows from the definition of $\calR$ as a norm and of $\topo$.
\item[\ref{asssum:Jtaulsc}]
like $\calJ=\calQ_A$ (cf. \eqref{estJb}), also $\calJKV$ is even $\topo$ continuous.
\item[\ref{asssum:Jcont},\ref{asssum:convJ}]
also here follows from independence of $\calJKV$ on the data $y=(\Volt,\Curr)$
\item[\ref{asssum:convR}]
again, due to  independence of $\calJKV$ on $y$ we only have to verify the second condition, which here  follows from 
\[
\begin{aligned}
&\calJKV(\sigma^\dagger,\Phi^\dagger,\Psi^\dagger)
-\calJKV(\tilde{\sigma}^\delta,\tilde{\Phi}^\delta,\tilde{\Psi}^\delta)\\
&=\calQ_A(\sigma^\dagger,\Phi^\dagger,\Psi^\dagger)
-\calQ_A(\tilde{\sigma}^\delta,\tilde{\Phi}^\delta,\tilde{\Psi}^\delta)\\
&\quad+\sum_{i=1}^I\int_{\partial\Omega}\phi^\dagger_i\nabla^\bot\psi^\dagger_i\cdot\nu
-\tilde{\phi}^\delta_i\nabla^\bot\tilde{\psi}^\delta_i\cdot\nu\, ds\\
&\leq 0+\sum_{i=1}^I\|\trace\phi^\dagger_i-\trace\tilde{\phi}^\delta_i\|_{L^\infty(\partial\Omega)}
\|\trace\tilde{\psi}^\delta_i\|_{W^{1,1}(\partial\Omega)}\\
&\quad+
\|\trace\phi^\dagger_i\|_{L^\infty(\partial\Omega)}
\|\trace\psi^\dagger_i-\trace\tilde{\psi}^\delta_i\|_{W^{1,1}(\partial\Omega)}\\
&\leq \tau\delta (\|\trace{\Psi}^\dagger\|_{W^{1,1}(\partial\Omega)}^{I}+\delta
+\|\trace\Phi^\dagger\|_{L^\infty(\partial\Omega)^I})
\end{aligned}
\]
cf. \eqref{intparts}.
\end{itemize}
\end{proof}

Note that for the last item we needed to use the $W^{1,1}(\partial\Omega)$ topology to bound the discrepancy of the boundary data for $\Psi$. Therefore formulation as a box constrained minimization problem like \eqref{EITbox} is not possible here.

\subsection{Convexity at the minimizer}\label{subsec:convexity}
A crucial prerequisite for fast convergence of minimization algorithms is positivity of the Hessian of the Lagrange function at the searched for minimizer.

Nonnegative definiteness of the Hessian of the Lagrangian on the critical cone is a necessary condition for a minimizer under certain assumptions on the cost functions and the constraints.  
For the original inverse problem with exact data \eqref{minJ}, this can be explicitely verified under appropriate differentiability assumptions on the involved operators. 

\revision{
We first of all consider the exact data case and restrict ourselves to functionals $\calS$, $\calQ_A$ defined by squared Hilbert space norms
\begin{equation}\label{Hilbertspacenorms}
\calS(y_1,y_2)=\frac12\|y_1-y_2\|_Y^2\,, \quad \calQ_A(x,u)=\frac12\|A(x,u)\|_W^2
\end{equation}
}

\paragraph{Reduced formulation}
For \eqref{minJred0} with $\calD=X$ (i.e., an unconstrained problem) we have, by $F(\xdag)=y$,
\[ 
\begin{aligned}
&\frac{D^2\calJ}{Dx^2}(\xdag)(h,h)=\|F'(\xdag)h\|_Y^2+(F(\xdag)-y,F''(\xdag)(h,h))_Y=\|F'(\xdag)h\|_Y^2\,.
\end{aligned}
\]

\paragraph{All-at-once formulation}
The same computation can be used for showing convexity of $\calJ(\cdot,\cdot,y)$ for \eqref{minJaao}, 
\revision{
with \eqref{Hilbertspacenorms} and $\bfF$ as in \eqref{Fxuy} in place of $F$.
}

\paragraph{Morozov type all-at-once formulation}
For \eqref{minJaaoM}, considering for simplicity only linear observations $C$, we have $\mathcal{L}(x,u,z)=\frac12\|A(x,u)\|_W^2+\langle z,Cu-y\rangle_{Y^*,Y}$, hence by $A(\xdag,\udag)=0$, 
\[
\hspace*{-1cm}
\begin{aligned}
&\frac{D^2\mathcal{L}}{D(x,u)^2}(\xdag,\udag,z)((h_x,h_u),(h_x,h_u))=
\|A'_x(\xdag,\udag)h_x+A'_u(\xdag,\udag)h_u\|^2\\
&+(A(\xdag,\udag),A''_{xx}(\xdag,\udag)(h_x,h_x)+2A''_{xu}(\xdag,\udag)(h_x,h_u)+A''_{uu}(\xdag,\udag)(h_u,h_u))_W\\
&=\|A'_x(\xdag,\udag)h_x+A'_u(\xdag,\udag)h_u\|^2\geq0
\end{aligned}
\]

Due to Remark \ref{rem:exKV-aaoM}, this also includes the variational form of EIT with $\calJ(\sigma,\Phi,\Psi)$ $=\frac12\|A(\sigma,\Phi,\Psi)\|_W^2$ as in \eqref{QAEIT}.

For $\calJ=\calJKV$ as in \eqref{JEIT}, we have 
\[
\mathcal{L}(\sigma,\Phi,\Psi,z_{\Volt},z_{\Curr})
=\calJKV(\sigma,\Phi,\Psi)+\sum_{i=1}^I\int_{\partial\Omega}\Bigl((\phi_i-\volt_i)z_{\Volt,i}
+(\psi_i-\curr_i)z_{\Curr,i}\Bigr)\, ds\,,
\]
hence, by $\nabla^\bot\psi_i^\dagger=\sigma^\dagger \nabla\phi_i^\dagger$,
\[
\begin{aligned}
&\frac{D^2\mathcal{L}}{D(x,u)^2}(\sigma^\dagger,\Phi^\dagger,\Psi^\dagger,z_{\Volt},z_{\Curr})(h_\sigma,h_\Phi,h_\Psi)^2\\
&= \sum_{i=1}^I\int_\Omega\Bigl(
\frac{h_\sigma^2}{(\sigma^\dagger)^3}|\nabla^\bot\psi_i^\dagger|^2
+\sigma^\dagger|\nabla h_{\Phi,i}|^2
+\frac{1}{\sigma^\dagger}|\nabla^\bot h_{\Psi,i}|^2\\
&\qquad\qquad+h_\sigma \nabla\phi_i^\dagger\cdot\nabla h_{\Phi,i}
-\frac{h_\sigma}{(\sigma^\dagger)^2} \nabla^\bot\psi_i^\dagger\cdot\nabla^\bot h_{\Psi,i}
\Bigr)\, dx\\
&= \sum_{i=1}^I\int_\Omega\Bigl(
\frac{1}{2\sigma^\dagger}\left|\frac{h_\sigma}{\sigma^\dagger}\nabla^\bot\psi_i^\dagger+\sigma^\dagger\nabla h_{\Phi,i}\right|^2
+\frac{1}{2\sigma^\dagger}\left|\frac{h_\sigma}{\sigma^\dagger}\nabla^\bot\psi_i^\dagger-\nabla^\bot h_{\Psi,i}\right|^2\\
&\qquad\qquad+\frac{\sigma^\dagger}{2}|\nabla h_{\Phi,i}|^2
+\frac{1}{2\sigma^\dagger}|\nabla^\bot h_{\Psi,i}|^2
\Bigr)\, dx\geq0
\end{aligned}
\]

\medskip

However, convexity cannot hold in a uniform sense here, since via known estimates under sufficient second order conditions, this would imply stability of the inverse problem.
This is why these convexity results do not allow to conclude convexity of the neighboring regularized problems. To achieve the latter, one has to impose either 
\begin{itemize}
\item[(i)] certain constraints on the nonlinearity of the involved operators and/on the regularization parameter, cf., e.g., \cite{ChaventKunisch,all-at-once},  i.e., conditions that are always satisfied in case of linear operators $A$, $C$ and $F$, but impose certain constraints on the nonlinearity structure of these operators in general. 
\item[(ii)] so--called source conditions, which in applications usually can be rephrased as regularity conditions on the exact solution, combined with some smoothness of the operators $A$, $C$ and $F$.
\end{itemize}  

\medskip

To investigate convexity in the noisy data case, we here consider the simple setting of the Ivanov bounds $\calRt(x)\leq\rho$, $\calRt(x,u)\leq\rho$ leading to linear constraints (like, e.g., in \eqref{EITbox}) and the additive regularization term being quadratic 
$\calR(x,u)=\frac12\|L((x,u)-(x_0,u_0)\|_Z^2$ for some possibly unbounded linear operator 
$L:X\times V\to Z$ and some Hilbert space $Z$. 
Then, for instance, convexity at $(\xad,\uad)$ in the respective cases can be checked under the following conditions, where (i) refers to a setting with structural restrictions on the nonlinearity, while under (ii) we consider source conditions.

\paragraph{Reduced formulation}
(i) For \eqref{minJRred} with $\calD=X$, let $F$ be twice differentiable and let the following restriction on nonlinearity of $F$ 
\begin{equation}\label{nonlincond_red}
\|F''(\xad)(h,h)\|\leq \frac{2}{\sqrt{2c+\|L(\xdag-x_0)\|_Z^2}}\|F'(\xad)h\|_Y \|Lh\|_Z 
\quad \forall h\in X
\end{equation}
with $c$ as in \eqref{alphared} hold. Then by minimality 
\begin{equation}\label{minimality_red}
\begin{aligned}
&\frac12\|F(\xad)-\ydel\|^2+\frac{\alpha}{2}\|L(\xad-x_0)\|_Z^2\\
&\leq \frac12\|F(\xdag)-\ydel\|^2+\frac{\alpha}{2}\|L(\xdag-x_0)\|_Z^2\\
&\leq \delta+\tfrac{\alpha}{2}\|L(\xdag-x_0)\|_Z^2\leq (c+\tfrac12\|L(\xdag-x_0)\|_Z^2)\ \alpha
\end{aligned}
\end{equation}
hence
\[ 
\begin{aligned}
&\frac{D^2\mathcal{L}}{Dx^2}(\xad)(h,h)=\|F'(\xad)h\|_Y^2+(F(\xad)-y,F''(\xad)(h,h))_Y+\alpha\|Lh\|_Z^2\\
&\geq 2\sqrt\alpha \|F'(\xad)h\|_Y \|Lh\|_Z-\|F(\xad)-\ydel\|_Y\|F''(\xad)(h,h)\|_Y\\
&\geq \sqrt\alpha \Bigl(2\|F'(\xad)h\|_Y \|Lh\|_Z-\sqrt{2c+\|L(\xdag-x_0)\|_Z^2}
\|F''(\xad)(h,h)\|\Bigr)\geq 0\\
& \forall h\in X\,.
\end{aligned}
\]
(ii) Under a variational source condition 
\begin{equation}\label{scred}
-(L(\xdag-x_0),L(x-x^\dagger))_Z\leq \frac{1}{2\bar{c}}\|F(x)-y\|_Y\quad x\in X
\end{equation}
where we again assume $F$ be twice differentiable with uniformly bounded second derivative in the sense that
\[
\sup_{h\in X}\frac{\|F''(x)(h,h)\|}{\|Lh\|_Z^2}=:\bar{c}<\infty
\]
we have, by minimality \eqref{minimality_red},
\[
\begin{aligned}
&\|F(\xad)-y\|_Y^2\leq 2\delta+\alpha(\|L(\xdag-x_0)\|_Z^2-\|L(\xad-x_0)\|_Z^2)\\
&=2\delta-2\alpha(L(\xdag-x_0),L(\xad-x^\dagger))-\alpha\|L(\xad-x^\dagger)\|_Z^2
\leq 2\delta+\alpha \frac{1}{\bar{c}}\|F(\xad)-y\|_Y\,,
\end{aligned}
\]
hence 
\[
\frac12 \|F(\xad)-y\|_Y^2+\frac12 (\|F(\xad)-y\|_Y-\frac{\alpha}{\bar{c}})^2
\leq2\delta+\frac12\frac{\alpha^2}{\bar{c}^2}
\]
i.e., for $\alpha\geq 2\bar{c}\sqrt{\delta}$ (which is compatible with \eqref{alphared} and actually corresponds to the optimal a priori choice $\alpha\sim\sqrt{\delta}$ under the variational source condition \eqref{scred}) we get
\[
\|F(\xad)-y\|_Y\leq \frac{\alpha}{\bar{c}}\,,
\]
so that we arrive at
\[ 
\begin{aligned}
&\frac{D^2\mathcal{L}}{Dx^2}(\xad)(h,h)=\|F'(\xad)h\|_Y^2+(F(\xad)-y,F''(\xad)(h,h))_Y+\alpha\|Lh\|_Z^2\\
&\geq \|F'(\xad)h\|_Y^2 -\|F(\xad)-y\|_Y\bar{c}\|Lh\|_Z^2 +\alpha\|Lh\|_Z^2
\geq \|F'(\xad)h\|_Y^2\geq0
\end{aligned}
\]

\paragraph{All-at-once formulation}
Analogously, with the notation \eqref{Fxuy}, we get the following two alternative sufficient conditions for convexity:
\begin{equation}\label{nonlincond_aao}
\begin{aligned}
(i)\quad&\|A''_{xx}(\xad,\uad)(h_x,h_x)+A''_{xu}(\xad,\uad)(h_x,h_u)+A''_{uu}(\xad,\uad)(h_u,h_u)\|_W^2\\
&\quad+\|C''(\uad)(h_u,h_u)\|_Y^2\\
&\leq \tfrac{4}{\|L((\xdag,\udag)-(x_0,u_0))\|_Z^2+2c}
\Bigl(\|A'_x(\xad,\uad)h_x+A'_u(\xad,\uad)h_u\|_W^2\\
&\qquad\qquad+\|C'(\uad)h_u\|_Y^2\Bigr)\|L(h_x,h_u)\|_Z^2 \\
&\forall (h_x,h_u)\in X\times V
\end{aligned}
\end{equation}
or 
\begin{equation}\label{scaao}
\begin{aligned}
(ii)\quad &(L((\xdag,\udag)-(x,u)),L((x,u)-(\xdag,\udag)))_Z\\
&\leq \frac{1}{\sqrt{2}\bar{c}}
(\|A(x,u)\|_W^2+\|C(u)-y\|_Y)\quad (x,u)\in X\times V
\end{aligned}
\end{equation}
where 
\[
\begin{aligned}
&\bar{c}^2:=
\sup_{h\in X}\frac{1}{\|L(h_x,h_u)\|_Z^2} \Bigl(\|C''(\uad)(h_u,h_u)\|_Y^2\\
&\|A''_{xx}(\xad,\uad)(h_x,h_x)+A''_{xu}(\xad,\uad)(h_x,h_u)+A''_{uu}(\xad,\uad)(h_u,h_u)\|_W^2\Bigr)\,.
\end{aligned}
\]

\paragraph{Morozov type all-at-once formulation}
The same reasoning could also be used in case \eqref{minJRaaoM} if for $\calS$ a discrepancy measure leading to linear constraints $B(x,u)-b\leq0$ with some linear operator $B:X\times V\to \tilde{Z}$ for some Banach space $\tilde{Z}$ (like, e.g., in \eqref{EITbox}) is used so that these constraints give no contribution to the Hessian of the Lagrange function. Then by minimality 
\[
\begin{aligned}
&\frac12\|A(\xad,\uad)\|_W^2+\frac{\alpha}{2}\|L((\xad,\uad)-(x_0,u_0))\|_Z^2\\
&\leq \frac12\|A(\xdag,\udag)\|^2+\frac{\alpha}{2}\|L((\xdag,\udag)-(x_0,u_0))\|_Z^2\\
&= \alpha\ \frac12\|L((\xdag,\udag)-(x_0,u_0))\|_Z^2\,,
\end{aligned}
\]
and under the condition 
\begin{equation}\label{nonlincond_aaoM}
\begin{aligned}
&\|A''_{xx}(\xad,\uad)(h_x,h_x)+A''_{xu}(\xad,\uad)(h_x,h_u)+A''_{uu}(\xad,\uad)(h_u,h_u)\|_W\\
&\leq \frac{2}{\|L((\xdag,\udag)-(x_0,u_0))\|_Z}
\|A'_x(\xad,\uad)h_x+A'_u(\xad,\uad)h_u\|_W\|L(h_x,h_u)\|_Z\\
&\forall (h_x,h_u)\in X\times V\,.
\end{aligned}
\end{equation}
we get convexity of the Lagrange function at $(\xad,\uad)$
\begin{equation}\label{convexLagraao} 
\begin{aligned}
&\frac{D^2\mathcal{L}}{D(x,u)^2}(\xad,\uad)((h_x,h_u),(h_x,h_u))=
\|A'_x(\xad,\uad)h_x+A'_u(\xad,\uad)h_u\|_W^2\\
&\quad+(A(\xad,\uad),A''_{xx}(\xad,\uad)(h_x,h_x)+A''_{xu}(\xad,\uad)(h_x,h_u)\\
&\qquad +A''_{uu}(\xad,\uad)(h_u,h_u))_W+\alpha\|L(h_x,h_u)\|_Z^2\\
&\geq 0
\end{aligned}
\end{equation}
Since $A'(\xad,\uad)$ typically has a nontrivial nullspace, condition \eqref{nonlincond_aaoM} is quite unlikely to hold, though.
Unfortunately a similar statement holds for the source condition 
\[
(L((\xdag,\udag)-(x_0,u_0)),L((x,u)-(\xdag,\udag)))_Z\leq c\|A(x,u)\|_W\quad (x,u)\in X\times V
\]
for some $c>0$.

In such a situation one might use the regularization parameter $\alpha$ to enforce convexity due to estimate \eqref{convexLagraao}, via a regularization parameter choice 
\begin{equation}\label{alphaaaoMadhoc}
\alpha\geq \bar{c} \|A(\xad,\uad)\|_W
\end{equation}
where
\begin{equation}\label{cbaraao}
\begin{aligned}
&\bar{c}\geq
\sup_{h\in X}\frac{1}{\|L(h_x,h_u)\|_Z^2}\\
&\quad\|A''_{xx}(\xad,\uad)(h_x,h_x)+A''_{xu}(\xad,\uad)(h_x,h_u)+A''_{uu}(\xad,\uad)(h_u,h_u)\|_W \,.
\end{aligned}
\end{equation}

Indeed, e.g., for $\calJ=\frac12\|A\|_W^2$ in Example \ref{exKV}, we can obtain an estimate of the form \eqref{cbaraao} as follows.
First of all, we estimate
\[
\begin{aligned}
&\|A''_{xx}(\xad,\uad)(h_x,h_x)+A''_{xu}(\xad,\uad)(h_x,h_u)+A''_{uu}(\xad,\uad)(h_u,h_u)\|_W\\
&=\Bigl(\sum_{i=1}^I\|
-\frac{h_\sigma^2}{4\sigma^{3/2}}\nabla\phi_i 
-\frac{3h_\sigma^2}{4\sigma^{5/2}}\nabla^\bot\psi_i
+\frac{h_\sigma}{2\sigma^{1/2}}\nabla h_{\phi_i}
+\frac{h_\sigma}{2\sigma^{3/2}}\nabla^\bot h_{\psi_i}\|_{L^2(\Omega)}^2
\Bigr)^{1/2}\\
&\leq 
\sum_{i=1}^I
\frac{1}{4\ul{\sigma}^{3/2}}
\|h_\sigma\|_{L^{4p/(p-1)}(\Omega)}^2
\|\phi_i\|_{W^{1,2p}(\Omega)}
+\frac{3}{4\ul{\sigma}^{5/2}}
\|h_\sigma\|_{L^{4p/(p-1)}(\Omega)}^2
\|\psi_i\|_{W^{1,2p}(\Omega)}\\
&\qquad+\frac{1}{2\ul{\sigma}^{1/2}}
\|h_\sigma\|_{L^{2p/(p-1)}(\Omega)}
\|h_{\phi_i}\|_{W^{1,2p}(\Omega)}
+\frac{1}{2\ul{\sigma}^{3/2}}
\|h_\sigma\|_{L^{2p/(p-1)}(\Omega)}
\|h_{\psi_i}\|_{W^{1,2p}(\Omega)}\\
&\leq \tilde{c} \|(h_\sigma,h_\Phi,h_\Psi)\|_{L^{4p/(p-1)}(\Omega)\times W^{1,2p}(\Omega)^{2I}}^2 
\end{aligned}
\]
Thus, using additional regularization of $\sigma$, i.e., 
\begin{equation}\label{REITsigma}
\calR(\sigma,\Phi,\Psi)=\|\sigma\|_{H^{1-\tilde{\epsilon}}(\Omega)}^2
+\|(\Phi,\Psi)\|_{H^{3/2-\epsilon}(\Omega)^{2I}}^2
=:\|L(\sigma,\Phi,\Psi)\|_{L^2(\Omega)^{2I+1}}^2\,,
\end{equation}
with 
\begin{equation}
\epsilon, \tilde{\epsilon}\in(0,\tfrac12), \quad \epsilon+2\tilde{\epsilon}\leq \tfrac12
\end{equation} 
in place of \eqref{REIT},
we obtain, using continuity of the embeddings $H^{1-\tilde{\epsilon}}(\Omega)\to$\\ $L^{4p/(p-1)}(\Omega)$ and $H^{3/2-\epsilon}(\Omega)\to W^{1,2p}(\Omega)$ with $p=\frac{2}{1+2\epsilon}$, due to $\epsilon+2\tilde{\epsilon}\leq \frac12$, that  
\[
\calR(h_\sigma,h_\Phi,h_\Psi)=\|L(h_\sigma,h_\Phi,h_\Psi)\|_{L^2(\Omega)}^2
\geq \ul{c}\|(h_\sigma,h_\Phi,h_\Psi)\|_{L^{4p/(p-1)}(\Omega)\times W^{1,2p}(\Omega)^{2I}}^2  
\]
for some $\ul{c}>0$. Hence we get \eqref{cbaraao} with $\bar{c}=\sqrt{\frac{\tilde{c}}{\ul{c}}}$. 

In principle, \eqref{alphaaaoMadhoc} can be implemented by using a fixed point iteration 
\[
\alpha^{k+1}=\bar{c} \|A(x_{\alpha^k}^\delta,u_{\alpha^k}^\delta)\|_W
\]
where each step involves computation of $(x_{\alpha^k}^\delta,u_{\alpha^k}^\delta)$, i.e.,  solution of the regularized minimization problem \eqref{minJR} with regularization parameter $\alpha^k$.
However, existence of a fixed point and its compatibility with the requirement $\alpha\to0$ as $\delta\to0$ from Corollary \ref{cor:convredaaoMaao} is yet to be investigated.

\section{Conclusions and outlook}
In this paper we have carried out some first steps into investigating minimization based formulation and regularization of inverse problems, motivated by the variational (Kohn-Vogelius) approach for EIT.
We have proven convergence for a large class of regularization methods, containing the existing theory on reduced and all-at-once regularization, but also new regularization approaches and in particular a bound constrained minimization approach for regularizing the variational formulation of EIT. 

Our next step in this context will be to employ the fast method devised in \cite{HungerlaenderRendl15} for solving the regularized EIT problem \eqref{EITbox} via a sequence of box constrained quadratic programs.\\ 
Further future work in this context will on one hand be concerned with extending the convergence analysis, e.g. to 
(a) rates under source conditions;
(b) the use of alternative choices of the Ivanov parameter $\rho$;
(c) iterative regularization by Newton or gradient (Landweber) type methods.\\
On the other hand we aim at extending the variational approach to other applications, such as identification of material parameters in linear magnetistatics or elastostatics or detection of cracks or inclusions from surface measurements, 
\revision{
cf. Examples \ref{exKV_mag}, \ref{exKV_crack}.
}

\section*{Acknowledgment}
The author wishes to thank Franz Rendl and Philipp Hun\-ger\-l\"ander, Alpen-Adria Universit\"at Klagenfurt, for fruitful discussions motivation the formulation and development the regularized EIT problem \eqref{EITbox}. 
Moreover, the author gratefully acknowledges financial support by the Austrian Science Fund FWF under the grants I2271 ``Regularization and Discretization of Inverse Problems for PDEs in Banach Spaces'' and P30054 ``Solving Inverse Problems without Forward Operators'' as well as partial support by the Karl Popper Kolleg ``Modeling-Simulation-Optimization'', funded by the Alpen-Adria-Universit\" at Klagenfurt and by the Carin\-thian Economic Promotion Fund (KWF).

Moreover, we wish to thank the reviewers for fruitful comments leading to an improved version of the manuscript.

\end{document}